\documentclass[a4paper, 11pt, reqno, twoside]{amsart}

\pdfoutput=1
\usepackage[dvipsnames]{xcolor}
\usepackage{lipsum}
\usepackage{graphicx} 
\usepackage{amssymb,amsfonts,amsmath, amsthm}
\usepackage{mathtools}
\usepackage{enumitem}
\usepackage{fancyhdr}
\usepackage{calc}
\usepackage{aligned-overset}
\usepackage{tabularx}
\usepackage{tabto}
\usepackage{csquotes}
\usepackage{url}
\usepackage{colonequals}
\usepackage{stmaryrd}
\usepackage{siunitx}
\usepackage{float}
\usepackage{import}
\usepackage{thmtools}
\usepackage{xifthen}
\usepackage{pdfpages}
\usepackage{transparent}
\usepackage{enumitem}
\usepackage{emptypage}
\usepackage[pass]{geometry} 
\usepackage{hyperref}
\usepackage{tikz}
\usepackage{pdfpages}

\allowdisplaybreaks[2] 

\newtheorem{thm}{Theorem}[section]
\newtheorem*{thm*}{Theorem}
\newtheorem{lem}[thm]{Lemma}
\newtheorem{prop}[thm]{Proposition}
\newtheorem{cor}[thm]{Corollary}
\newtheorem{ques}{Question}
\theoremstyle{definition}
\newtheorem{defi}[thm]{Definition}

\newtheorem*{rem}{Remark}

\newcommand{\Z}{\mathbb{Z}}
\newcommand{\crn}{\mathrm{cr}} 
\newcommand{\ngamma}[1]{\crn(\gamma_1, \dots, \gamma_{#1})}
\newcommand{\ngammai}[1]{\crn(\gamma_{i_1}, \dots, \gamma_{i_{#1}})}
\newcommand{\gammak}[1]{\gamma_1, \dots, \gamma_{#1}}
\newcommand{\gammai}[1]{\gamma_{i_1}, \dots, \gamma_{i_{#1}}}
\newcommand{\ntorus}[1]{\crn{\left( \left \lceil \textstyle \frac{#1}{2} \right \rceil;1 \right )}+\crn{\left( \left \lfloor \textstyle \frac{#1}{2} \right \rfloor;1 \right )}}
\newcommand{\ndelta}[1]{\crn(\delta_1, \dots, \delta_{#1})}

\newcommand{\slz}{\mathrm{SL}_2(\mathbb{Z})}
\newcommand{\sgnm}{\text{sgn}}

\graphicspath{ {./figures/} }
\definecolor{Redcurve}{RGB}{255,0,0}
\definecolor{Orangecurve}{RGB}{255,128,0}
\definecolor{Greencurve}{RGB}{0,178,0}
\definecolor{Bluecurve}{RGB}{34,0,204}
\definecolor{Purplecurve}{RGB}{166,0,255}

\title{Crossing Number of Curves on Surfaces}
\author{Jasmin Jörg}
\date{}

\begin{document}

\begin{abstract}
We consider systems of simple closed curves on surfaces and their total number of intersection points, their so-called crossing number. For a fixed number of curves, we aim to minimise the crossing number. We determine the minimal crossing number of up to 12 curves on a surface of genus 2 and prove that minimising systems are unique up to homeomorphisms of the surface and isotopies of curves. 
\end{abstract}

\maketitle

\tableofcontents


\section{Introduction}

In this article, the objects of study are systems of simple closed curves on surfaces and their total number of intersection points, their so-called crossing number. Optimising such systems of curves has long been of interest, usually taking the approach of maximising the number of curves under given conditions on the intersection numbers. It is, for example, well-known that in a closed surface of genus $g$, there are at most $3g-3$ pairwise disjoint and non-homotopic, simple closed curves, maximal systems being given by a decomposition of the surface into pairs of pants. It is also known that there are at most $2g+1$ curves intersecting pairwise exactly once \cite{Malestein_2013}. However, the maximal number of curves in so-called 1-systems, that is, systems of curves intersecting pairwise at most once, remains an open question for $g\geq 3$. The answer is easily given for a torus, where a maximal 1-system consists of three curves. For a surface of genus 2, the answer is given by Malestein, Rivin and Theran. They prove that maximal 1-systems contain 12 curves and that exactly two mapping class orbits of maximal 1-systems exist \cite{Malestein_2013}. For higher genus surfaces, the exact answer remains unknown; the best upper bound is of order $g^2 \log(g)$, given by Greene \cite{greene2018curves}. 

Besides being fascinating objects in their own right, 1-systems are of interest due to their connection to systoles in hyperbolic geometry: In a closed hyperbolic surface, any two systoles intersect at most once. Thus, systems of systoles naturally are 1-systems. The maximal number of systoles among hyperbolic surfaces of a given genus, the so-called kissing number, is therefore closely related to the study of 1-systems. For genus 2, it is known that the kissing number is 12, which is realised by a unique hyperbolic surface, the Bolza surface \cite{schmutz1994systoles}. The best known upper bound for the kissing number of surfaces of genus $g$ is of order $\frac{g^2}{\log(g)}$, given by Parlier \cite{Parlier_2013}.
 
In this article, we take the reverse approach; we fix the number of curves $k$ and aim to determine how many intersection points any $k$ pairwise non-homotopic curves necessarily create. That is, what is the minimal crossing number of $k$ curves in a surface of genus $g$?
Our interest is twofold: On the one hand, we determine the minimal crossing number of $k$ curves on a genus $g$ surface, denoted $\crn(k;g)$, 
for small $k$, focusing mainly, but not exclusively, on a surface of genus 2. 
On the other hand, we consider the properties of minimising systems and see that while they are unique for small systems and small genus, this seems not to be the case in general. 
Our main goal is to prove the following theorem.

\begin{thm} \label{thm:mainthm}
In a surface of genus 2, the minimal crossing numbers of up to 12 curves are:
    \begin{align*}
        \crn(4;2)&=1, & \crn(7;2)&=6, & \crn(10;2)&=21, \\
        \crn(5;2)&=2, & \crn(8;2)&=10, & \crn(11;2)&=28, \\
        \crn(6;2)&=4, & \crn(9;2)&=14, & \crn(12;2)&=36.
    \end{align*}
For $4\leq k \leq 12$, realisations of $\crn(k;2)$ are unique up to homeomorphisms of the surface and isotopies of curves. Further, for $k \leq 11$, any minimal system of $k$ curves may be obtained by adding a curve to a minimal system of $k-1$ curves. The unique minimal system of 12 curves cannot be obtained in this way.
\end{thm}

One of the curve systems we examine more closely is a system of 12 curves on a surface of genus 2 that arises from a description of a maximal 1-system given in \cite{Malestein_2013}. It is related to the system of systoles in the Bolza surface via a homeomorphism of the surface and isotopies of curves. One of our goals is to prove that it is minimal in terms of the crossing number of 12 curves. We further prove that any minimising system of 12 curves is a 1-system, which in turn implies that minimal systems of 12 curves are unique. 

The main part of this article is devoted to the proof of Theorem~\ref{thm:mainthm}, which is a collection of several results contained in Sections~\ref{section:counting} and \ref{section:genus2}: The minimal crossing numbers are given in Propositions~\ref{prop:smallkgeneralg}, \ref{prop:genus2_89}, \ref{prop:genus2_1011} and \ref{prop:genus2_12}, the uniqueness of minimal systems is given in Proposition~\ref{prop:uniqueness11} and Corollary~\ref{cor:uniqueness12}.
Along the way, we obtain some more general results; namely, we determine $\crn(k;g)$ for $g \geq 2$ and $k\leq 5g-3$ in Section~\ref{section:counting} as well as $\crn(k;1)$ for $k\leq 6$ in Section~\ref{section:genus1}. Realisations of minimal crossing numbers in higher genus surfaces are investigated in Section~\ref{section:higher genus}. We explore the order of growth of $\crn(k;g)$ in Section~\ref{section:order of growth}, where we give a lower bound of order $k^2$ and an upper bound of order $k^{2+\frac{1}{3g-3}}$.


\section{Crossing number of curves} \label{section:intersection number}

Let $\Sigma_g$ be a smooth, closed, connected and orientable surface of genus $g$. 
A \emph{simple closed curve} in $\Sigma_g$ is an embedded circle, that is, a closed curve with no self-intersections. Any simple closed curve is isotopic to a smooth simple closed curve, we may therefore assume all curves to be smooth. Further, any two curves that intersect are isotopic to curves that intersect transversely, we thus assume that all intersections are transverse. In particular, this means the number of intersection points of any two simple closed curves is finite.

For simple closed curves $\gamma$ and $\delta$, there is a notion of an intersection number: the \emph{geometric intersection number} of $\gamma$ and $\delta$, denoted $i(\gamma, \delta)$, is the minimal number of intersection points between curves $\tilde{\gamma}$ and $\tilde{\delta}$ homotopic to $\gamma$ and $\delta$:
\[i(\gamma, \delta) = \min \{\#(\tilde{\gamma} \cap \tilde{\delta}) \mid  \tilde{\gamma} \sim \gamma, \tilde{\delta} \sim \delta\},\]
where $\sim$ denotes homotopy between curves. When $i(\gamma, \delta) = \#(\gamma \cap \delta)$, we say that $\gamma$ and $\delta$ are in \emph{minimal position}. For details, see e.g.\ Chapter 1 in \cite{farb2012primer}.

For a collection of simple closed curves $\gamma_1, \dots, \gamma_k$ in $\Sigma_g$, we consider the \emph{crossing number} of $\gamma_1, \dots, \gamma_k$, defined as
\[\crn(\gamma_1, \dots, \gamma_k) = \sum_{1\leq i < j \leq k} i(\gamma_i, \gamma_j).\]
When convenient, we write $\crn(\Gamma)$ for $\ngamma{k}$ if $\Gamma=\{\gammak{k}\}$ is a finite set of curves.

We define $\crn(k;g)$ as the 
minimum of all crossing numbers $\ngamma{k}$, where $\gammak{k} \subset \Sigma_g$ is any collection of $k$ non-contractible, pairwise non-homotopic, simple closed curves. We call $\crn(k;g)$ the \emph{minimal crossing number} of $k$ curves in $\Sigma_g$. 
We refer to a system of curves $\gammak{k}$ with $\ngamma{k}=\crn(k;g)$ as a \emph{minimal system} or a \emph{realisation of $\crn(k;g)$}.

Throughout Chapters~\ref{section:counting}-\ref{section:outlook}, whenever we consider a system of curves, it is to be understood that these curves are non-contractible, pairwise non-homotopic, simple, closed, and with transverse intersection. Unless otherwise stated, we assume that all curves are in minimal position.


\section{Counting intersection points} \label{section:counting}

We start out with a few simple topological considerations, which let us determine $\crn(k;g)$ for $k\leq 5g-3$. After that, we establish a fundamental counting argument that will be applied numerous times throughout the following chapters. 

\begin{prop} \label{prop:smallkgeneralg}
    Let $g \geq 2$. Then 
     \begin{samepage}
    \begin{enumerate}[label=(\roman*)]
        \item $\crn(k;g) = 0$, for $k\leq 3g-3$, \label{smallk_item1}
        \item $\crn(3g-3+k;g) = k$, for $1 \leq k \leq g$, \label{smallk_item2}
        \item $\crn(4g-3+k;g) = g+2k$, for $1\leq k \leq g$. \label{smallk_item3}
    \end{enumerate}
     \end{samepage}
\end{prop}
\begin{proof}
    \ref{smallk_item1}: This is an immediate consequence of the fact that a surface of genus $g$ may be decomposed into pairs of pants by $3g-3$ pairwise disjoint simple closed curves. \\
    \ref{smallk_item2}: Clearly $\crn(k;g)$ is increasing in $k$, we prove that for $k\geq 3g-3$, it is strictly increasing. Assume towards a contradiction that there exists ${k \geq 3g-3}$ such that $\crn(k+1;g)=\crn(k;g)$ and let $\Gamma = \{\gammak{k+1} \}$ be a minimal system of curves in $\Sigma_g$. For all $i \in \{1, \dots, k+1\}$, the set $\Gamma_i = \Gamma \setminus \{ \gamma_i \}$ is a system of $k$ curves, so $\crn(\Gamma_i) \geq \crn(k;g)$. We get
    \begin{align*}
        \sum_{j=1}^{k+1} i(\gamma_i, \gamma_j) &= \crn(\Gamma) - \crn(\Gamma_i)\\
        &\leq \crn(k+1;g) - \crn(k;g) \\
        &= 0,
    \end{align*}
    so $\sum_{j=1}^{k+1} i(\gamma_i, \gamma_j) =0$ for all $i \in \{1, \dots, k+1 \}$. But this means that the $k+1$ curves in $\Gamma$ are pairwise disjoint, contradicting $k+1 > 3g-3$. Thus, $\crn(k;g)$ is strictly increasing for $k\geq 3g-3$ and therefore $\crn(k;g) \geq k-(3g-3)$ for all $k\geq 3g-3$ or equivalently
    \begin{equation}
        \crn(3g-3+k;g) \geq k \label{eq:linearlower}
    \end{equation}
    for all $k \geq 0$. This establishes one inequality of \ref{smallk_item2}. For the other inequality, consider the following system of curves (see Figure~\ref{fg:3g-3+k}): Let $\alpha_1, \dots, \alpha_{2g-3}$ be pairwise disjoint curves such that $\Sigma_g \setminus (\alpha_1 \cup \dots \cup \alpha_{2g-3})$ consists of $g$ tori with one removed disc and $g-2$ pairs of pants. In each torus, choose two curves $\beta_i$ and $\gamma_i$ intersecting exactly once. For $1\leq k \leq g$, set 
    $\Delta_k = \{ \alpha_1, \dots, \alpha_{2g-3}, \beta_1, \dots, \beta_g, \gamma_1, \dots, \gamma_k\}$.
    We have:
    \begin{align*}
        \crn(3g-3+k;g) &\leq \crn(\Delta_k) \\
        &= k.
    \end{align*}
    \ref{smallk_item3}: Above, we have shown that $\crn(k;g)$ is strictly increasing in $k$ for ${k\geq 3g-3}$, which means that $\crn(k+1;g) \geq \crn(k;g)+1$. We prove that for $k\geq 4g-3$, we have $\crn(k+1;g) \geq \crn(k;g)+2$, which together with $\crn(4g-3;g)=g$ implies one inequality of \ref{smallk_item3}: $\crn(4g-3+k;g) \geq g+2k$. For this, assume towards a contradiction that $\crn(k+1;g) = \crn(k;g)+1$ for some $k\geq 4g-3$ and let $\Gamma = \{ \gammak{k+1} \}$ be a minimal system. For $i\in \{1, \dots, k+1\}$, consider $\Gamma_i = \Gamma \setminus \{ \gamma_i\}$, which is a system of $k$ curves and therefore satisfies $\crn(\Gamma_i) \geq \crn(k;g)$. It follows that
    \begin{align*}
        \sum_{j=1}^{k+1} i(\gamma_i, \gamma_j) &= \crn(\Gamma)-\crn(\Gamma_i) \\
        &\leq \crn(k+1;g) - \crn(k;g) \\
        &= 1,
    \end{align*}
    so each curve $\gamma_i$ intersects at most one other curve, and if it does intersect another curve, there is exactly one intersection point. This means that $\Gamma$ contains $\crn(k+1;g)$ disjoint pairs of curves intersecting exactly once, and all other intersection numbers are zero. In a genus $g$ surface, there are at most $g$ such pairs of curves, contradicting $\crn(k+1;g) > g$, which follows from (\ref{eq:linearlower}) since $k\geq 4g-3$. Thus, $\crn(k+1;g) \geq \crn(k;g)+2$ for $k\geq 4g-3$.
    
    For the other inequality, consider once more the system of curves in Figure~\ref{fg:3g-3+k}. For $1 \leq i \leq g$, let $\delta_i$ be the curve obtained from $\gamma_i$ by applying a Dehn twist about $\beta_i$. Set $\tilde{\Delta}_k = \{\alpha_1, \dots, \alpha_{2g-3}, \beta_1, \dots, \beta_g, \gamma_1, \dots, \gamma_g, \delta_1, \dots, \delta_k\}$. We have
    \begin{align*}
        \crn(4g-3+k;g) &\leq \crn(\tilde{\Delta}_k) \\
        &= g+2k. \qedhere
    \end{align*}
\end{proof}

\begin{figure}[ht] \centering 
        \includegraphics[width=0.7\textwidth]{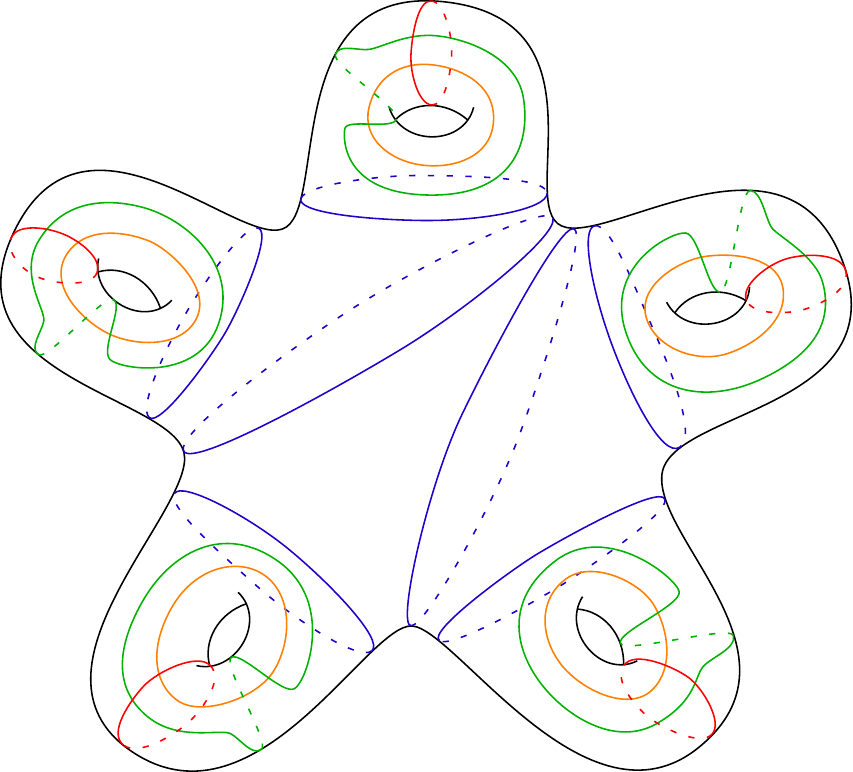}
        \caption{Curves \textcolor{Bluecurve}{$\alpha_1, \dots, \alpha_{2g-3}$,} \textcolor{Redcurve}{$\beta_1, \dots, \beta_g$,} \textcolor{Orangecurve}{$\gamma_1, \dots, \gamma_g$,} \textcolor{Greencurve}{$\delta_1, \dots, \delta_g$}}
        \label{fg:3g-3+k}
\end{figure}

\begin{lem} \label{lem:count}
    Let $g\geq 1$ and let $\gamma_1, \dots , \gamma_k \subset \Sigma_g$ be a system of curves satisfying for some $m \in \{2, \dots, k\}$ and $n\geq0$: $\crn(\gamma_{i_1}, \dots, \gamma_{i_m}) \geq n$ for any $1\leq i_1 < \dots < i_m \leq k$. Then
    \[\crn(\gamma_1, \dots, \gamma_k) \geq \left\lceil \frac{k(k-1)}{m(m-1)} n \right\rceil.\]
\end{lem}
\begin{proof}
    By assumption, $\ngammai{m} \geq n$ for any $1\leq i_1 < \dots < i_m \leq k$. Summing over all possible $1\leq i_1 < \dots < i_m \leq k$, each pairwise intersection number is counted $\binom{k-2}{m-2}$ times, thus,
    \begin{align*}
         \binom{k-2}{m-2} \crn(\gamma_1, \dots, \gamma_k) &= \sum_{1\leq i_1 < \dots < i_m \leq k} \crn(\gamma_{i_1}, \dots, \gamma_{i_m})    \\
         &\geq \binom{k}{m} n,
    \end{align*}
    and therefore
    \[ \crn(\gamma_1, \dots, \gamma_k) \geq \frac{k(k-1)}{m(m-1)} n.\]
    Since $\crn(\gammak{k})$ takes integer values, we may round up to the next integer on the right-hand side of this inequality.
\end{proof}

\begin{cor} \label{cor:count}
    For $k \geq 2$ and $2 \leq m \leq k$:
    \[\crn(k;g) \geq \left \lceil \frac{k(k-1)}{m(m-1)} \crn(m;g) \right \rceil.\]
\end{cor}


\section{Genus 1} \label{section:genus1}

When considering minimal systems on a surface of genus 2 in Chapter~\ref{section:genus2}, subsystems of curves occur that are contained in a subsurface of genus 1. Therefore, we briefly restrict our attention to systems of curves on the torus $\Sigma_1$.
We determine $\crn(k;1)$ for $k\leq 6$ and prove that for $k\leq 5$, realisations of the minimal crossing number are unique up to homeomorphisms of $\Sigma_1$ and isotopies of curves. We start by recalling some facts about simple closed curves on the torus (Chapters 1 and 2 in \cite{farb2012primer}). 

There is a bijective correspondence between non-trivial isotopy classes of (oriented) simple closed curves on the torus and primitive vectors in $\mathbb{Z}^2$. For simplicity, we thus consider elements of $\Z^2$ rather than curves themselves. In a slightly imprecise use of language, we call elements $(a,b) \in \Z^2$ curves when we mean the (isotopy classes of) curves in $\Sigma_1$ corresponding to $(a,b)$. The geometric intersection number of two curves $(a,b)$ and $(c,d)$ is given by
\[i((a,b),(c,d)) = |ad-bc|.\]
Further, the mapping class group of the torus is isomorphic to $\text{SL}_2(\mathbb{Z})$, where elements $A \in \slz$ act on $\Z^2$ by left-multiplication. Note that we look at unoriented curves only, so $(a,b)$ and $(-a,-b)$ correspond to the same unoriented simple closed curve.
 
We consider the following curves:
\begin{align}\begin{split} \label{curves_torus}
    w_1 &= (1,0), \\
    w_2 &= (0,1), \\
    w_3 &= (1,1), \\
    w_4 &= (-1,1), \\
    w_5 &= (2,1), \\
    w_6 &= (1,2). 
    \end{split}
\end{align}
Computing the intersection numbers, we obtain the following upper bounds:
\begin{align} \begin{split} \label{upperbounds:torus}
    \crn(1; 1) &\leq \crn(w_1) = 0, \\
    \crn(2; 1) &\leq \crn(w_1, w_2) = 1, \\
    \crn(3; 1) &\leq \crn(w_1, w_2, w_3) = 3, \\
    \crn(4; 1) &\leq \crn(w_1, \dots, w_4) = 7, \\
    \crn(5; 1) &\leq \crn(w_1, \dots, w_5) = 14, \\
    \crn(6; 1) &\leq \crn(w_1, \dots, w_6) = 24.
    \end{split}
\end{align}
For $k\leq 6$, the curves $w_1, \dots, w_k$ minimise the crossing number of $k$ curves on the torus. As we see in the next lemma, this is immediate for $k\leq 4$. Further, these realisations are unique in the following sense.

\begin{lem} \label{lem:torus_leq4}
    For $k\leq 4$, the system $w_1, \dots, w_k$ given in (\ref{curves_torus}) is minimal, i.e.\ 
     $\crn(1;1)=0$, $\crn(2;1)=1$, $\crn(3;1)=3$, $\crn(4;1)=7$.
    Moreover, for any other minimal system $v_1, \dots, v_k \in \mathbb{Z}^2$, there exists $A\in \slz$ such that $\{Av_1, \dots, Av_k\} = \{w_1, \dots, w_k\}$.
\end{lem}
\begin{proof}
    We first prove that the systems are minimal. Trivially, $\crn(1;1)=0$. On the torus, any two non-homotopic simple closed curves intersect at least once, so $\crn(2;1) \geq 1$ and $\crn(3;1)\geq 3$. Moreover, a straightforward computation shows that a system of curves intersecting pairwise exactly once consists of at most three curves, so in a system of four curves, one of the intersection numbers must be at least two, yielding $\crn(4;1) \geq 7$. 
    
    For uniqueness, let $v_i = (a_i, b_i) \in \mathbb{Z}^2$, $i=1,\dots,k$, be distinct homotopy classes of curves such that $\crn(v_1, \dots, v_k) = \crn(k;1)$.\\
    For $k=1$, uniqueness is an immediate consequence of the fact that $(a_1, b_1)$ is primitive. Indeed, $\slz$ acts transitively on primitive vectors in $\Z^2$. \\
    $k=2$: Since $|a_1 b_2 - a_2 b_1|=1$, one of 
    $\begin{pmatrix}
        b_2 & -a_2 \\
        -b_1 & a_1
    \end{pmatrix}$, 
    $\begin{pmatrix}
        b_1 & -a_1 \\
        -b_2 & a_2
    \end{pmatrix}$
    is an element of $\slz$. Both map $\{v_1, v_2\}$ to $\{w_1, w_2\}$.\\
    $k=3$: The images of $\{v_1, v_2, v_3\}$ under the maps from the case $k=2$, are $\{w_1, w_2, w_3\}$ and $\{w_1, w_2, w_4\}$. If the latter, apply 
    $\begin{pmatrix}
        1 & 1 \\
        0 & 1
    \end{pmatrix}$, which maps $\{w_1, w_2, w_4\}$ to $\{w_1, w_2, w_3\}$. \\
    $k=4$: We may assume that $i(v_3, v_4)=2$ and all other intersection numbers are 1. In particular, this means $\crn(v_1,v_2,v_3)=3$, so by the case $k=3$, after a homeomorphism (and possibly renaming $v_1, v_2, v_3$) we may assume that $v_1=(1,0)$, $v_2=(0,1)$, $v_3=(1,1)$, $v_4=(a,b)$. We get 
    \[4 = \sum_{i=1}^3 i(v_i, v_4) = |a|+|b|+|a-b|.\] 
    Therefore, $v_4 \in \{(-1,1), (2,1), (1,2)\}=\{w_4,w_5,w_6\}$. If $v_4 = (2,1)$, apply 
    $\begin{pmatrix}
        1 & -1\\
        0 & 1
    \end{pmatrix}$, 
    if $v_4=(1,2)$, apply 
    $\begin{pmatrix}
        1 & 0\\
        -1 & 1
    \end{pmatrix}$.
\end{proof}

In order to find the minimal crossing number of 5 or 6 curves, we first consider the possible intersection patterns of four curves. For this, we introduce the intersection graph.

\begin{defi}
    The \emph{intersection graph} of curves $v_1, \dots, v_k \subset \Sigma_1$ is the 
    weighted graph on vertices $\{1,\dots,k\}$, where for $1\leq i, j \leq k$, there is an edge of weight $i(v_i, v_j)$ between vertices $i$ and $j$.
\end{defi}
For simplicity, we leave edges of weight 1 unlabeled.

\begin{lem}\label{lem:torus_4curves}
    Let $v_1, v_2, v_3, v_4 \in \Z^2$ be any collection of curves. Then $\crn(v_1, \dots, v_4) \neq 8$. Moreover, if $\crn(v_1, \dots, v_4) =9$, then the intersection graph of $v_1, \dots, v_4$ is isomorphic to:
    \begin{center}
    \emph{
    \begin{tikzpicture}[main/.style = {draw, circle}] 
    \node[main] (1) {$1$}; 
    \node[main] (2) [right of=1] {$2$};
    \node[main] (3) [below of=1] {$3$}; 
    \node[main] (4) [right of=3] {$4$};
    \draw (1) -- (2);
    \draw (1) -- (3);
    \draw (1) -- (4);
    \draw (2) -- (3);
    \draw (2) -- node[midway, right, pos=0.5] {3} (4);
    \draw (3) -- node[midway, below, pos=0.5] {2} (4);
    \end{tikzpicture}
    }
    \end{center}
\end{lem}
\begin{proof} 
    Suppose $\crn(v_1,\dots, v_4)=8$. We consider all possible intersection graphs of $v_1, \dots, v_4$. Up to isomorphism (which corresponds to renaming $v_1, \dots, v_4$), the intersection graph is one of the following:
    \begin{center}
    \begin{tikzpicture}[main/.style = {draw, circle}] 
    \node[main] (1) {$1$}; 
    \node[main] (2) [right of=1] {$2$};
    \node[main] (3) [below of=1] {$3$}; 
    \node[main] (4) [right of=3] {$4$};
    \node[main] (5) [right of=2] {$1$};
    \node[main] (6) [right of=5] {$2$};
    \node[main] (7) [below of=5] {$3$};
    \node[main] (8) [right of=7] {$4$};
    \node[main] (9) [right of=6] {$1$};
    \node[main] (10) [right of=9] {$2$};
    \node[main] (11) [below of=9] {$3$};
    \node[main] (12) [right of=11] {$4$};
    \draw (1) -- (2);
    \draw (1) -- (4);
    \draw (2) -- (3);
    \draw (2) -- (4);
    \draw (5) -- (6);
    \draw (5) -- (8);
    \draw (6) -- (7);
    \draw (7) -- (8);
    \draw (9) -- (10);
    \draw (9) -- (12);
    \draw (10) -- (11);
    \draw (10) -- (12);
    \draw (11) -- (12);
    \draw (1) -- node[midway, left, pos=0.5] {2} (3);
    \draw (3) -- node[midway, below, pos=0.5] {2} (4);
    \draw (5) -- node[midway, left, pos=0.5] {2} (7);
    \draw (6) -- node[midway, right, pos=0.5] {2} (8);
    \draw (9) -- node[midway, left, pos=0.5] {3} (11);
    \end{tikzpicture}
    \end{center}
    In particular, $i(v_1, v_2)=1$. By Lemma \ref{lem:torus_leq4}, after a homeomorphism we may assume that $v_1=(1,0)$, $v_2=(0,1)$, $v_3=(a,b)$, $v_4=(c,d)$. We may further assume that $b,d >0$. We have: $i(v_1, v_3) = b$, $i(v_2, v_3) = |a|$, $i(v_1, v_4) = d$, $i(v_2, v_4) = |c|$. Hence,
    \begin{align}
        i(v_3, v_4) &= |ad-bc| \nonumber\\
        &= \begin{cases}
            |i(v_2,v_3)i(v_1,v_4) - i(v_1,v_3)i(v_2,v_4)| &\text{ if $\sgnm(a)=\sgnm(c)$,}\\
            |i(v_2,v_3)i(v_1,v_4) + i(v_1,v_3)i(v_2,v_4)| &\text{ if $\sgnm(a)\neq \sgnm(c)$.}
        \end{cases} \label{eq:intersectiongraph}
    \end{align}
    Looking at the intersection graphs above, we see that none of them satisfies either of these equations, contradicting the assumption $\crn(v_1, \dots, v_4)=~8$.

    Now assume that $\crn(v_1, \dots, v_4)=9$. Up to isomorphism, the possible intersection graphs are the following:
    \begin{center}
    \begin{tikzpicture}[main/.style = {draw, circle}] 
    \node[main] (1) {$1$}; 
    \node[main] (2) [right of=1] {$2$};
    \node[main] (3) [below of=1] {$3$}; 
    \node[main] (4) [right of=3] {$4$};
    \node[main] (5) [right of=2] {$1$};
    \node[main] (6) [right of=5] {$2$};
    \node[main] (7) [below of=5] {$3$};
    \node[main] (8) [right of=7] {$4$};
    \node[main] (9) [right of=6] {$1$};
    \node[main] (10) [right of=9] {$2$};
    \node[main] (11) [below of=9] {$3$};
    \node[main] (12) [right of=11] {$4$};
    \node[main] (13) [below of=3] {$1$}; 
    \node[main] (14) [right of=13] {$2$};
    \node[main] (15) [below of=13] {$3$}; 
    \node[main] (16) [right of=15] {$4$};
    \node[main] (17) [right of=14] {$1$};
    \node[main] (18) [right of=17] {$2$};
    \node[main] (19) [below of=17] {$3$};
    \node[main] (20) [right of=19] {$4$};
    \node[main] (21) [right of=18] {$1$};
    \node[main] (22) [right of=21] {$2$};
    \node[main] (23) [below of=21] {$3$};
    \node[main] (24) [right of=23] {$4$};
    \draw (1) -- (2);
    \draw (1) -- node[midway, left, pos=0.5] {2} (3);
    \draw (1) -- node[midway, below, pos=0.65] {2} (4);
    \draw (2) -- (3);
    \draw (2) -- (4);
    \draw (3) -- node[midway, below, pos=0.5] {2} (4);
    \draw (5) -- (6);
    \draw (5) -- node[midway, left, pos=0.5] {2} (7);
    \draw (5) -- (8);
    \draw (6) -- node[midway, below, pos=0.65] {2} (7);
    \draw (6) -- (8);
    \draw (7) -- node[midway, below, pos=0.5] {2} (8);
    \draw (9) -- (10);
    \draw (9) -- node[midway, left, pos=0.5] {2} (11);
    \draw (9) -- (12);
    \draw (10) -- (11);
    \draw (10) -- node[midway, right, pos=0.5] {2} (12);
    \draw (11) -- node[midway, below, pos=0.5] {2} (12);
    \draw (13) -- (14);
    \draw (13) -- (15);
    \draw (13) -- (16);
    \draw (14) -- (15);
    \draw (14) -- node[midway, right, pos=0.5] {3} (16);
    \draw (15) -- node[midway, below, pos=0.5] {2} (16);
    \draw (17) -- (18);
    \draw (17) -- node[midway, left, pos=0.5] {2} (19);
    \draw (17) -- (20);
    \draw (18) -- (19);
    \draw (18) -- node[midway, right, pos=0.5] {3} (20);
    \draw (19) -- (20);
    \draw (21) -- (22);
    \draw (21) -- node[midway, left, pos=0.5] {4} (23);
    \draw (21) -- (24);
    \draw (22) -- (23);
    \draw (22) -- (24);
    \draw (23) -- (24);
    \end{tikzpicture}
    \end{center}
    The same procedure as above reveals that the only intersection graph satisfying one of \eqref{eq:intersectiongraph} is the one at the bottom left, which is the one given in the lemma. Note that this intersection graph is realised e.g.\ by $v_1=(0,1)$, $v_2=(1,0)$, $v_3=(1,1)$, $v_4=(1,3)$.
\end{proof}

\begin{lem} \label{lem:torus_5curves}
The system of curves $w_1, \dots, w_5$ given in (\ref{curves_torus}) is optimal, i.e.\ $\crn(5;1) = 14$. Moreover, for any minimal system $v_1, \dots, v_5$, there exists $A\in \slz$ such that $\{Av_1, \dots, Av_5\} = \{w_1, \dots, w_5\}$.
\end{lem}
\begin{proof}
    Suppose $v_1, \dots, v_5 \in \Z^2$ is a minimal system of five curves, i.e.\ $\crn(v_1, \dots, v_5)=\crn(5;1)$. Then there exist $1\leq i_1 < i_2 < i_3 < i_4 \leq 5$ such that $\crn(v_{i_1},\dots, v_{i_4}) = 7$, since otherwise, by Lemma~\ref{lem:count} and Lemma~\ref{lem:torus_4curves}, $\crn(5;1)=\crn(v_1, \dots, v_5) \geq  \left \lceil \frac{5 \cdot 9}{3} \right \rceil=15$, contradicting the upper bound given in (\ref{upperbounds:torus}).  Thus, we may assume that $\crn(v_1, \dots, v_4)=7$ and, by Lemma~\ref{lem:torus_leq4}, that $v_1=(1,0)$, $v_2=(0,1)$, $v_3=(1,1)$, $v_4=(-1,1)$, $v_5=(a,b)$. We have:
    \begin{align}
        \sum_{i=1}^{4} i(v_i, v_5) &= |b|+|a|+|b-a|+|a+b| \nonumber\\
        &= 3 \max\{|a|,|b|\}+\min\{|a|,|b|\} \nonumber\\
        &\geq 7 \label{eq:torus_geq7},
    \end{align}
    where the last inequality holds because $a,b \neq 0$ and $\max\{|a|,|b|\} \geq 2$. It follows that 
    \begin{align*}
        \crn(5;1) &= \crn(v_1, \dots, v_5) \\
        &= \crn(v_1, \dots, v_4) + \sum_{i=1}^{4} i(v_i, v_5) \\
        &\geq 14.
    \end{align*}
    For uniqueness, observe that if $\crn(v_1, \dots, v_5)=14$, we get equality in (\ref{eq:torus_geq7}), implying that $v_5 \in \{(2,1), (-2,1), (1,2), (-1,2)\}$. If $v_5 = (-2,1)$, apply 
    $\begin{pmatrix}
        1 & 1\\
        0 & 1
    \end{pmatrix}$,
    if $v_5=(1,2)$, apply
    $\begin{pmatrix}
        1 & -1\\
        1 & 0
    \end{pmatrix}$,
    if $v_5=(-1,2)$, apply
    $\begin{pmatrix}
        0 & -1\\
        1 & 0
    \end{pmatrix}$.
\end{proof}

\begin{lem}
    The system of curves $w_1, \dots, w_6$ given in (\ref{curves_torus}) is optimal, i.e.\ $\crn(6;1) = 24$.
\end{lem}
\begin{proof}
    Suppose $v_1, \dots, v_6 \in \Z^2$ is a minimal system of 6 curves. We consider two cases: First, suppose there exist $1\leq i_1 < \dots < i_4 \leq 6$ such that $\crn(v_{i_1}, \dots , v_{i_4})=7$. We may assume w.l.o.g.\ that $\{v_{i_1}, \dots , v_{i_4}\}=\{v_1, \dots, v_4\}$ and, by Lemma~\ref{lem:torus_leq4}, we may further assume that $v_1=(1,0)$, $v_2=(0,1)$, $v_3=(1,1)$, $v_4=(-1,1)$, $v_5=(a_5,b_5)$, $v_6=(a_6,b_6)$. As in (\ref{eq:torus_geq7}), for $j=5,6$, we have
    \begin{align}
        \sum_{i=1}^{4} i(v_i, v_j) &= 3\max\{|a_j|, |b_j|\} + \min\{|a_j|,|b_j|\} \nonumber \\
        &\geq 7, \label{eq:7}
    \end{align}
    with equality if and only if $v_j \in A=\{(2,1), (1,2), (-2,1), (-1,2)\}$. In fact, if this sum is not equal to 7, then it is at least 10, implying equality in (\ref{eq:7}) for $j=5,6$ since 
    \begin{align*}
        \sum_{i=1}^4 i(v_i, v_5) + \sum_{i=1}^4 i(v_i, v_6) &= \underbrace{\crn(v_1, \dots, v_6)}_{\leq 24} - \underbrace{\crn(v_1, \dots, v_4)}_{=7} - \underbrace{i(v_5, v_6)}_{\geq 1} \\
        & \leq 16.
    \end{align*}
    Now observe that $i(v,w) \geq 3$ for any $v,w \in A$, thus,
    \begin{align*}
        \crn(6;1) &= \crn(v_1, \dots, v_6) \\
        &= \crn(v_1, \dots, v_4) + i(v_5, v_6) +\sum_{j=5}^{6} \sum_{i=1}^{4} i(v_i, v_j) \\
        &\geq 7+3+2\cdot7 \\
        &= 24.        
    \end{align*}
    
    For the second case, assume that there are no four curves among $v_1, \dots, v_6$ with crossing number 7. By Lemma~\ref{lem:torus_4curves}, that means any four curves intersect at least 9 times. There must exist $1\leq i_1 < \dots < i_4 \leq6$ such that $\crn(v_{i_1}, \dots , v_{i_4})=9$, since otherwise, by Lemma~\ref{lem:count}, $\crn(v_1, \dots, v_6) \geq \left \lceil \frac{6\cdot 5}{4\cdot 3} 10 \right \rceil=25$, contradicting $\crn(6;1)\leq 24$. We may thus assume that $\crn(v_1, \dots, v_4)=9$, and that the intersection graph of $v_1, \dots, v_4$ is as in Lemma~\ref{lem:torus_4curves}. In particular, this means $\crn(v_1, v_2, v_3)=3$. By assumption, $\crn(v_1, v_2, v_3, v_j) \geq 9$ for $j=4,5,6$, therefore
    \begin{align*}
        \sum_{i=1}^{3} i(v_i, v_j) &= \crn(v_1, v_2, v_3, v_j) - \crn(v_1, v_2, v_3) \\
        &\geq 9-3 =6.
    \end{align*}
    This implies that
    \begin{align*}
        \crn(6;1) &= \crn(v_1, \dots, v_6) \\
        &= \crn(v_1, v_2, v_3) + \crn(v_4,v_5,v_6) +\sum_{j=4}^{6} \sum_{i=1}^{3} i(v_i, v_j) \\
        &\geq 3+3+3\cdot 6 \\
        &=24. \qedhere
    \end{align*}
\end{proof}

\begin{rem}
    With the same kind of reasoning as in Lemma~\ref{lem:torus_5curves}, it is possible to prove that realisations of $\crn(6;1)$ are unique up to homeomorphisms and isotopies of curves.
\end{rem}


\section{Genus 2} \label{section:genus2}

We now focus on our main surface of interest, $\Sigma_2$. For simpler notation, throughout this chapter, we will use $\crn(k)$ to denote $\crn(k;2)$. Our main goal is to determine the minimal crossing number of up to 12 curves. It will transpire that with the exception of very small systems, that is, systems of up to three curves, minimal systems are unique up to homeomorphisms of $\Sigma_2$ and isotopies of curves. 

Recall from Proposition~\ref{prop:smallkgeneralg} that trivially $\crn(1)=\crn(2)=\crn(3)=0$ and further:
\begin{align}
\begin{split}
    \crn(4)= 1, \\
    \crn(5) = 2, \\
    \crn(6) = 4, \\
    \crn(7) = 6. \label{eq:cr_kleq7}
\end{split}
\end{align}
In the proof of Proposition~\ref{prop:smallkgeneralg}, we constructed realisations of the above minimal crossing numbers. They correspond to subsystems of the curves $\delta_1, \dots, \delta_7$ given in Figure~\ref{fg:genus2_11curves}. Recall also that the minimal systems given in the proof of that proposition were constructed inductively, i.e.\ they were obtained from the minimal system of $k-1$ curves by adding a curve with as few intersection points as possible. At first glance, it seems plausible that this behaviour is true in general, in this chapter we will see that this is not the case: While this turns out to be true for minimal systems of up to 11 curves, the unique minimal system of 12 curves cannot be obtained in this fashion. In fact, the realisation of $\crn(12)$ looks remarkably different from the smaller optimal systems.

\subsection{Minimal crossing number of up to 11 curves}

We consider the curves $\delta_1, \dots, \delta_{11}$ in Figure~\ref{fg:genus2_11curves}. Our primary goal is to prove for $8 \leq k \leq 11$, that $\delta_1, \dots, \delta_{k}$ is a minimal system.

    \begin{figure}[ht] \centering 
        \includegraphics[width=0.9\textwidth]{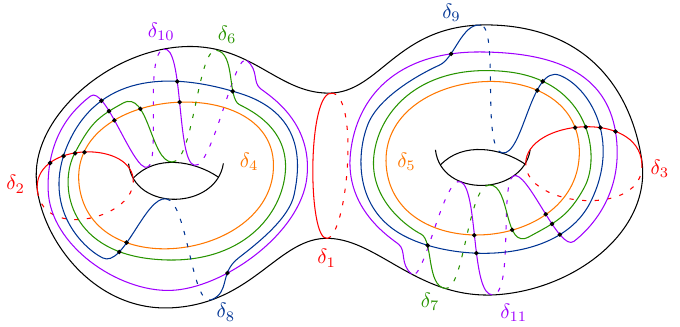}
        \caption{Curves $\delta_1, \dots, \delta_{11}$
        }
        \label{fg:genus2_11curves}
    \end{figure}
The curves in the figure are in minimal position, so the geometric intersection numbers may be determined by counting intersection points.
We have:
\begin{align} \label{upperbounds:nk}
\begin{split}
    \crn(8) &\leq \ndelta{8} = 10, \\
    \crn(9) &\leq \ndelta{9} = 14, \\
    \crn(10) &\leq \ndelta{10} = 21, \\
    \crn(11) &\leq \ndelta{11} = 28.  
\end{split}
\end{align}

In order to prove that these systems are optimal, we treat the cases ${k=8,9}$ and $k=10,11$ separately, needing stronger results for the latter. In both cases, we proceed as follows: First, we give lower bounds for the crossing number of curve systems that do not contain any three pairwise disjoint curves. We apply these results to prove the existence of a separating curve in minimal systems. We then find lower bounds for $\crn(k)$ depending on the number of curves intersecting the separating curve. This approach not only yields the minimality of the above systems but also provides enough information about the structure of minimal systems that uniqueness is an almost immediate consequence.

\begin{lem}\label{lem:nopop}
    For $k\geq 3$, let $\gamma_1, \dots , \gamma_k \subset \Sigma_2$ be a system of curves such that for any $1\leq i_1 < i_2 < i_3 \leq k$, the curves $\gamma_{i_1}, \gamma_{i_2}, \gamma_{i_3}$ are not pairwise disjoint. Then
    \[\crn(\gamma_1, \dots , \gamma_k) \geq \left \lceil \frac{k(k-2)}{4} \right \rceil.\]
\end{lem}
\begin{proof}
    As in Chapter~\ref{section:genus1}, we consider the intersection graph of $ \gamma_1, \dots, \gamma_k$. It is sufficient to consider the unweighted intersection graph, i.e.\ the graph with vertices $V=\{1, \dots, k\}$ and edges $E= \{\{i,j\} \mid i,j \in V \text{ and } i(\gamma_i, \gamma_j)\geq 1\}$. Since there is no triple of pairwise disjoint curves among $\gamma_1, \dots, \gamma_k$, the complement of the intersection graph is triangle-free. Thus, by Turán's theorem (Theorem 12.3 in \cite{bondy2008graph}), the complement contains at most $\left \lfloor \frac{k^2}{4} \right \rfloor$ edges, yielding the desired lower bound for the crossing number:
    \begin{align}
        \crn(\gamma_1, \dots, \gamma_k) &\geq \#E \label{eq:nrofedges}\\
        &\geq \binom{k}{2}-\left \lfloor \frac{k^2}{4} \right \rfloor \label{eq:turanstheorem}\\
        &= \left \lceil \frac{k(k-2)}{4} \right \rceil.\nonumber \qedhere
    \end{align}
\end{proof}

\begin{cor} \label{cor:pop}
    For $4 \leq k\leq 9$, any minimal system $\gammak{k} \subset \Sigma_2$ contains three pairwise disjoint curves, i.e.\ there exist $1\leq i_1 < i_2 < i_3 \leq k$ such that $\crn(\gamma_{i_1}, \gamma_{i_2}, \gamma_{i_3})=0$. Moreover, one of $\gamma_{i_1}, \gamma_{i_2}, \gamma_{i_3}$ is separating.
\end{cor}
\begin{proof}
    The existence of pairwise disjoint curves $\gamma_{i_1}, \gamma_{i_2}, \gamma_{i_3}$ follows from the upper bounds for $\crn(k)$ given in \eqref{eq:cr_kleq7} and \eqref{upperbounds:nk} together with Lemma~\ref{lem:nopop}. To see that one of the curves must be separating, observe that if all of them are non-separating, then $i(\gamma_j, \gamma_{i_1})+i(\gamma_j, \gamma_{i_2})+i(\gamma_j, \gamma_{i_3}) \geq 2$ for $j \neq i_1,i_2,i_3$. Therefore, 
        \[ \ngamma{k} \geq 2(k-3),\]
    contradicting the values of $\crn(k)$ in \eqref{eq:cr_kleq7} for $4\leq k \leq 7$. For $k=8,9$, use that the remaining curves $\{\gamma_i \mid i\neq i_1,i_2,i_3\}$ intersect at least $\crn(k-3)$ times to get
    \[\ngamma{k} \geq 2(k-3) + \crn(k-3).\]
    We know $\crn(5)=2$ and $\crn(6)=4$, so the above lower bound contradicts the upper bound given in (\ref{upperbounds:nk}).
\end{proof}

The above corollary guarantees the existence of a non-contractible separating curve $\delta$ in minimal systems of $k$ curves for $4\leq k \leq 9$. The two connected components of $\Sigma_2 \setminus \delta$ are homeomorphic to a punctured torus. We would like to give a lower bound for the crossing number of curves contained in $\Sigma_2 \setminus \delta$, by viewing them as a curve system in the disjoint union of two tori and applying the results from the previous chapter. For this, we need that non-homotopic curves in a connected component of $\Sigma_2 \setminus \delta$ remain non-homotopic when considered as curves in a torus, which we will prove in the following lemma. As a consequence, the number of intersection points is bounded below by the minimal crossing number for the torus.

\begin{lem}
    Let $p \in \Sigma_1$ and let $\alpha, \beta \subset \Sigma_1 \setminus \{p\}$ be two non-contractible simple closed curves, not homotopic to $p$. Then $\alpha$ and $\beta$ are homotopic in $\Sigma_1 \setminus \{p\}$ if and only if they are homotopic in $\Sigma_1$.
\end{lem}
\begin{proof}
    Clearly, if $\alpha$ and $\beta$ are homotopic in $\Sigma_1 \setminus \{p\}$, they are also homotopic in $\Sigma_1$. For the converse, suppose that $\alpha \sim \beta$ in $\Sigma_1$. Apply isotopies in $\Sigma_1 \setminus \{p\}$ and obtain curves $\tilde{\alpha} \sim \alpha$ and $\tilde{\beta} \sim \beta$  such that $\tilde{\alpha}$ and $\tilde{\beta}$ do not bound a bigon in $\Sigma_1 \setminus \{p\}$. We will prove that there is also no bigon between $\tilde{\alpha}$ and $\tilde{\beta}$ in $\Sigma_1$. For this, assume that there is such a bigon. Consider $\Sigma_1 \setminus \tilde{\alpha}$, which is homeomorphic to an annulus since $\tilde{\alpha}$ is non-trivial and not homotopic to $p$. As $\tilde{\alpha}$ and $\tilde{\beta}$ bound a bigon, in particular they intersect, so $\tilde{\beta} \cap (\Sigma_1 \setminus \tilde{\alpha})$ is a collection of disjoint simple arcs with endpoints on the boundary. Bigons in $\Sigma_1$ between $\tilde{\alpha}$ and $\tilde{\beta}$ correspond to arcs starting and ending on the same boundary component. Observe that the number of endpoints on each boundary component must coincide, and therefore there exists a second bigon between $\tilde{\alpha}$ and $\tilde{\beta}$. But at most one of those bigons contains $p$, so at least one of them is a bigon in $\Sigma_1 \setminus \{p\}$, giving a contradiction. Hence, $\tilde{\alpha}$ and $\tilde{\beta}$ do not bound a bigon in $\Sigma_1$, so by the bigon criterion (see e.g.\ Proposition 1.7 in \cite{farb2012primer}) they are in minimal position in $\Sigma_1$. This means that $\tilde{\alpha} \cap \tilde{\beta} = \emptyset$ since $\tilde{\alpha} \sim \tilde{\beta}$ in $\Sigma_1$. Thus, $\Sigma_1 \setminus (\tilde{\alpha} \cup \tilde{\beta})$ is the disjoint union of two annuli, only one of which contains $p$. Hence, $\tilde{\alpha} \sim \tilde{\beta}$ and therefore $\alpha \sim \beta$ in $\Sigma_1 \setminus \{p\}$.
\end{proof}

\begin{cor}\label{cor:genus1togenus2}
    Let $\delta, \gamma_1, \dots, \gamma_k \subset \Sigma_2$ be a system of curves such that $\delta$ is separating and $\gammak{k} \subset \Sigma_2 \setminus \delta$. Then 
    \[\ngamma{k} \geq \min \left\{\crn{\left(k-l;1\right)}+\crn(l;1) \mid l=0,\dots,k \right\}.\]
\end{cor}

\begin{lem}\label{lem:mainestimation}
    Let $k\leq 12$ and let $\gammak{k} \subset \Sigma_2$ be a system of simple closed curves such that $\gamma_1$ is separating. Set $m =\#\{\gamma_i \mid 2\leq i \leq k\text{, } i(\gamma_1, \gamma_i) \geq 1\}$. Then
    \begin{enumerate}[label = (\roman*)]
        \item $\ngamma{k} \geq 2m + \crn(k-1)$, \label{eq:mlarge}
        \item $\ngamma{k} \geq \crn(m) + m(k-m-1) + \ntorus{k-m-1}$. \label{eq:k-m-1}
    \end{enumerate}
\end{lem}
\begin{proof}
    For easier notation, let $\Gamma = \{\gamma_2, \dots, \gamma_k\}$, ${M=\{\gamma_i \in \Gamma \mid i(\gamma_1, \gamma_i) \geq 1\}}$, so that $m=\#M$. \\
    \ref{eq:mlarge}: Since $\gamma_1$ is separating, any other curve has an even number of intersection points with $\gamma_1$, in particular, $i(\gamma_1, \gamma_i)\geq2$ for any $\gamma_i \in M$. This yields
    \begin{align}
        \ngamma{k} &= \sum_{i=2}^{k} i(\gamma_1, \gamma_i) + \crn(\Gamma) \nonumber \\
        &= \sum_{\gamma_i \in M} i(\gamma_1, \gamma_i) + \crn(\Gamma) \nonumber \\
        &\geq 2m + \crn(\Gamma).\label{eq:rewrite1}
    \end{align}
    Clearly $\crn(\Gamma)\geq \crn(k-1)$, which implies \ref{eq:mlarge}.\\
    \ref{eq:k-m-1}: We further rewrite $\crn(\Gamma)$:
    \begin{align}
        \crn(\Gamma) &= \crn(M) + \crn(\Gamma \setminus M) + \sum_{\gamma_i \in M} \sum_{\gamma_j \in \Gamma \setminus M} i(\gamma_i, \gamma_j).\label{eq:rewrite2}
    \end{align}
    We give lower bounds for all terms individually. First, we consider $\Gamma \setminus M$. The curves in $\Gamma \setminus M$ are contained in $\Sigma_2 \setminus \gamma_1$, we may therefore apply Corollary~\ref{cor:genus1togenus2}. We have
    \begin{align*}
        \crn(\Gamma \setminus M) \geq \min \left\{\crn{\left(\#(\Gamma \setminus M)-l;1\right)}+\crn(l;1) \mid l=0,\dots,\#(\Gamma \setminus M) \right\}.
    \end{align*}
    Using the values of $\crn(k;1)$ for $k\leq 6$ together with $\crn(7;1)\geq \left \lceil \frac{7 \cdot 24}{5} \right \rceil = 34$ and $\crn(8;1) \geq \left \lceil \frac{8 \cdot 34}{6} \right \rceil = 46$, given by Corollary~\ref{cor:count}, a simple computation shows that the above minimum is attained for $l= \left \lceil \frac{\#(\Gamma \setminus M)}{2} \right \rceil$, i.e.\ when the curves in $\Gamma \setminus M$ are split as evenly as possible between the two connected components of $\Sigma_2 \setminus \gamma_1$. Noting that $\#(\Gamma \setminus M) = k-m-1$, we get:
     \[ \crn(\Gamma \setminus M) \geq \ntorus{k-m-1}.\]
    Next, let us consider any curve $\gamma_i \in M$. The curve $\gamma_i$ may be disjoint from one curve in $\Gamma \setminus M$ in each connected component of $\Sigma_2 \setminus \gamma_1$, but must intersect all others at least once. This can be seen as follows: The closure of each connected component of $\Sigma_2 \setminus \gamma_1$ is homeomorphic to a torus with one boundary component. Since $\gamma_i \cap \gamma_1 \neq \emptyset$, the restriction of $\gamma_i$ to one of the connected components corresponds to at least one (non-trivial) simple arc with endpoints on the boundary. The complement of this arc in that connected component is homeomorphic to an annulus, thus contains exactly one simple closed curve. This is the only curve in that connected component that may be disjoint from $\gamma_i$. It follows that
    \[\sum_{\gamma_j \in \Gamma\setminus M} i(\gamma_i, \gamma_j) \geq \#(\Gamma \setminus M) - 2 = k-m-3,\] 
    and therefore,
    \[\sum_{\gamma_i \in M} \sum_{\gamma_j\in \Gamma \setminus M} i(\gamma_i, \gamma_j) \geq m(k-m-3).\]
    Finally, $\crn(M) \geq \crn(m)$. Together with \eqref{eq:rewrite1} and \eqref{eq:rewrite2} these inequalities yield \ref{eq:k-m-1}.
\end{proof}

\begin{prop}\label{prop:genus2_89}
    For $k = 8,9$, the system of curves $\delta_1, \dots, \delta_k$ given in Figure~\ref{fg:genus2_11curves} is optimal, i.e.\ \begin{align*}
        \crn(8)&=\ndelta{8}=10, \\
        \crn(9)&=\ndelta{9}=14.
    \end{align*}
\end{prop}
\begin{proof}
    Let $\gamma_1, \dots, \gamma_k$ be such that $\ngamma{k}=\crn(k)$. By Corollary~\ref{cor:pop}, we may assume that $\gamma_1$ is separating. Set ${m =\#\{\gamma_i \mid 2\leq i \leq k \text{, } i(\gamma_1, \gamma_i) \geq 1\}}$.\\
    $k=8$: If $m \geq 3$, then, by Lemma~\ref{lem:mainestimation}~\ref{eq:mlarge},
    \begin{align*}
        \crn(8) &= \crn(\gammak{8}) \\
        &\geq 2m+\crn(7) \\
        &\geq 12,
    \end{align*}
    contradicting the upper bound $\crn(8)\leq10$ given in (\ref{upperbounds:nk}). Therefore, $m\leq 2$ and, using Lemma~\ref{lem:mainestimation}~\ref{eq:k-m-1}, we obtain
    \begin{align*}
        \crn(8) &= \crn(\gammak{8}) \\
        &\geq
        \begin{cases}
            \crn(3;1)+\crn(4;1) &\text{ if $m=0$,} \\
            6+2\crn(3;1) &\text{ if $m=1$,} \\
            2\cdot 5 + \crn(2;1)+\crn(3;1) &\text{ if $m=2$,}
        \end{cases}
        \\
        &= 
        \begin{cases}
            10 &\text{ if $m=0$,} \\
            12 &\text{ if $m=1$,} \\
            14 &\text{ if $m=2$.}
        \end{cases}
    \end{align*}
    For $m=1$ and $m=2$, this is a contradiction to the known upper bound, so $m=0$ and $\crn(8) = 10$.
    \\
    $k=9$: Again, if $m \geq 3$, then 
    \begin{align*}
        \crn(9) &= \crn(\gammak{9}) \\
        &\geq 2m+\crn(8) \\
        &\geq 16,
    \end{align*}
    by Lemma~\ref{lem:mainestimation}~\ref{eq:mlarge}. This contradicts $\crn(9)\leq14$, thus, $m\leq 2$. An application of Lemma~\ref{lem:mainestimation}~\ref{eq:k-m-1} yields
    \begin{align*}
        \crn(9) &= \crn(\gammak{9}) \\
        &\geq
        \begin{cases}
            2\crn(4;1) &\text{ if $m=0$,} \\
            7+\crn(3;1)+ \crn(4;1) &\text{ if $m=1$,} \\
            2\cdot 6 + 2\crn(3;1) &\text{ if $m=2$,}
        \end{cases}
        \\
        &= 
        \begin{cases}
            14 &\text{ if $m=0$,} \\
            17 &\text{ if $m=1$,} \\
            18 &\text{ if $m=2$.}
        \end{cases}
    \end{align*}
    For $m=1$ and $m=2$, this contradicts the known upper bound, so $m=0$ and $\crn(9) = 14$.
\end{proof}

In order to establish that the systems $\gammak{k}$ are minimal for $k=10,11$, we need to improve the lower bound given in Lemma~\ref{lem:nopop} for the crossing number of $k$ curves not containing any three pairwise disjoint curves. 

\begin{lem}\label{lem:nopop89}
    Let $\gammak{9} \subset \Sigma_2$ be curves such that no three curves among them are pairwise disjoint. Then 
    \begin{enumerate}[label = (\roman*)]
        \item $\ngamma{8} \geq 14$, \label{eq8:lem:nopop89}
        \item $\ngamma{9} \geq 21$. \label{eq9:lem:nopop89}
    \end{enumerate}
\end{lem}
\begin{proof}
    We prove (ii), the proof of (i) is analogous. Set $\Gamma = \{\gamma_1, \dots, \gamma_9\}$. We consider two cases. First, assume that there are four distinct curves $\gamma_{i_1}, \dots, \gamma_{i_4} \in \Gamma$ such that $i(\gamma_{i_1}, \gamma_{i_2}), i(\gamma_{i_3}, \gamma_{i_4})\geq 1$ and $i(\gamma_{i_j}, \gamma_{i_k}) =0$ for ${j=1,2}$ and $k=3,4$. Let us call this property ($\star$). Observe that ($\star$) implies that $\gamma_{i_1}, \dots, \gamma_{i_4}$ are non-separating. Thus, $\Sigma_2 \setminus (\gamma_{i_1} \cup \gamma_{i_3})$ is homeomorphic to a sphere with four punctures, and $\gamma_{i_2}, \gamma_{i_4}$ correspond to (nonempty) collections of simple arcs, each connecting two of the punctures. Hence, one of the connected components of $\Sigma_2 \setminus (\gamma_{i_1} \cup \dots \cup \gamma_{i_4})$ is an annulus. Let $\delta$ be a non-trivial simple closed curve contained in that annulus. Observe that by construction, $\delta \subset \Sigma_2$ is a separating curve disjoint from $\gamma_{i_1}, \dots, \gamma_{i_4}$ and is not homotopic to any of the curves in $\Gamma$. Set $M=\{\gamma \in \Gamma \mid i(\gamma, \delta) \neq 0\}$ and $m=\#M$. By construction of $\delta$, $m \leq 5$. We adapt the proof of Lemma~\ref{lem:mainestimation} slightly to obtain a lower bound for the crossing number depending on $m$. We start by rewriting the crossing number:
    \begin{align*}
        \ngamma{9} = \crn(M) + \crn(\Gamma \setminus M) + \sum_{\gamma_i \in M} \sum_{\gamma_j \in \Gamma \setminus M} i(\gamma_i, \gamma_j).
    \end{align*}
    By assumption, there are no three pairwise disjoint curves among $\gammak{9}$, therefore,
    \begin{align}
        \crn(M) \geq 
        \begin{cases}
            0 &\text{ if $m \leq 2$,} \\
            \left\lceil \frac{m(m-2)}{4} \right\rceil &\text{ if $m \geq 3$,}
        \end{cases} \label{eq:nM}
    \end{align}
    by Lemma~\ref{lem:nopop}.
    Moreover, the curves in $\Gamma \setminus M$ are contained in $\Sigma_2 \setminus \delta$, so we may apply Corollary~\ref{cor:genus1togenus2}. Together with the values of $\crn(k;1)$ determined in the previous chapter, this gives rise to the following lower bound for the crossing number:
    \begin{align*}
        \crn(\Gamma \setminus M) \geq  \ntorus{9-m}.
    \end{align*}
    Since there are no three pairwise disjoint curves in $\Gamma$, any $\gamma_i\in M$ intersects all but possibly one curve of $\Gamma \setminus M$, i.e.\
    \begin{align*}
        \sum_{\gamma_i \in M} \sum_{\gamma_j \in \Gamma \setminus M} i(\gamma_i, \gamma_j) &\geq m(8-m).
    \end{align*}
    Finally, using all of the above inequalities, we obtain:
    \begin{align*}
        \ngamma{9} &\geq \crn(M) + \ntorus{9-m} + m(8-m) \\
        & \overset{\mathclap{\eqref{eq:nM}}}{\geq} 
        \begin{cases}
            \crn(5; 1) + \crn(4; 1) &\text{ if $m=0$,} \\
            2\crn(4; 1) + 7 &\text{ if $m=1$,} \\
            \crn(4; 1) + \crn(3; 1) + 12 &\text{ if $m=2$,} \\
            1 + 2\crn(3; 1) + 15 &\text{ if $m=3$,} \\
            2+ \crn(3; 1) + \crn(2; 1) + 16 &\text{ if $m=4$,} \\
            4 + 2\crn(2; 1) + 15 &\text{ if $m=5$,} 
        \end{cases} \\
        &= 
        \begin{cases}
            21 &\text{ if $m=0$,} \\
            21 &\text{ if $m=1$,} \\
            22 &\text{ if $m=2$,} \\
            22 &\text{ if $m=3$,} \\
            22 &\text{ if $m=4$,} \\
            21 &\text{ if $m=5$,} 
        \end{cases} \\
        &\geq 21.
    \end{align*}
    For the second case, assume that there are no four curves in $\Gamma$ satisfying ($\star$). Consider any choice of $1\leq i_1 < \dots < i_4 \leq 9$. Looking at possible intersection graphs of $\gammai{4}$, it is easy to see that $\ngammai{4} \geq 3$ since any intersection graph with less than 3 edges contradicts either the assumption that $\Gamma$ contains no three pairwise disjoint curves or the assumption that $\gammai{4}$ do not satisfy ($\star$). Applying Lemma~\ref{lem:count} recursively, we get
    \begin{align*}
        \ngammai{5} &\geq \left \lceil \frac{5 \cdot 3}{3} \right \rceil = 5, \text{ for all $1 \leq i_1 < \dots < i_5 \leq9$,}\\
        \ngammai{6} &\geq \left \lceil \frac{6 \cdot 5}{4} \right \rceil = 8, \text{ for all $1 \leq i_1 < \dots < i_6 \leq9$,}\\
        \ngammai{7} &\geq \left \lceil \frac{7 \cdot 8}{5} \right \rceil = 12, \text{ for all $1 \leq i_1 < \dots < i_7 \leq9$,}\\
        \ngammai{8} &\geq \left \lceil \frac{8 \cdot 12}{6} \right \rceil = 16, \text{ for all $1 \leq i_1 < \dots < i_8 \leq9$,}
    \end{align*}
    and therefore
    \begin{equation*}
        \ngamma{9} \geq \left \lceil \frac{9 \cdot 16}{7} \right \rceil =21. \qedhere
    \end{equation*}
\end{proof}

\begin{cor}\label{cor:nopop_101112}
    Let $\gammak{12} \subset \Sigma_2$ be curves such that no three curves among them are pairwise disjoint. Then 
    \begin{enumerate}[label = (\roman*)]
        \item $\ngamma{10} \geq 27$,
        \item $\ngamma{11} \geq 33$,
        \item $\ngamma{12} \geq 40$.
    \end{enumerate}
\end{cor}
\begin{proof}
    Use Lemma~\ref{lem:nopop89} \ref{eq9:lem:nopop89} and apply Lemma~\ref{lem:count} three times.
\end{proof}

For $k\geq 10$, the existence of a separating curve in a minimal system is no longer given by Corollary~\ref{cor:pop}. However, a closer look at the crossing number of systems of non-separating curves reveals that this is still true for $k=10, 11$.

\begin{lem}\label{lem:nonsep}
    For $k = 10, 11$, any minimal system of $k$ curves in $\Sigma_2$ contains a separating curve.
\end{lem}
\begin{proof}
    For $k=10,11$, let $\gammak{k} \subset \Sigma_2$ be a minimal system. Recall that $\crn(10) \leq 21$ and $\crn(11) \leq 28$. By Corollary~\ref{cor:nopop_101112} we may therefore assume that $\gamma_1, \gamma_2, \gamma_3$ are pairwise disjoint. Assume towards a contradiction that all curves are non-separating. Since $\gamma_1, \gamma_2, \gamma_3$ are pairwise disjoint and non-separating, $i(\gamma_1, \gamma_i)+i(\gamma_2, \gamma_i)+i(\gamma_3, \gamma_i)\geq 2$ for $i=4,\dots, k$ and therefore
    \begin{equation}
        \ngamma{k} \geq 2(k-3) + \crn(\gamma_4, \dots , \gamma_k) \label{eq:nonsep_pop}.
    \end{equation}
    We have a closer look at the crossing number of $\gamma_4, \dots, \gamma_k$.
    \\
    $k=10$: We know $\ngamma{10} = \crn(10) \leq 21$, so inequality \eqref{eq:nonsep_pop} implies
    \begin{equation}
        \crn(\gamma_4, \dots, \gamma_{10}) \leq 7. \label{eq:remaining7}
    \end{equation}
    By Lemma~\ref{lem:nopop}, we may therefore assume that $\gamma_4, \gamma_5, \gamma_6$ are pairwise disjoint. But the same argument as above yields 
    \begin{align*}
        \crn(\gamma_4, \dots , \gamma_{10}) &\geq 2\cdot 4 + \crn(\gamma_7, \dots, \gamma_{10}) \\
        &\geq 8+\crn(4) \\
        &= 9,
    \end{align*}
    which is a contradiction to \eqref{eq:remaining7}. \\
    $k=11$: Since $\ngamma{11} = \crn(11) \leq 28$, from inequality \eqref{eq:nonsep_pop} it follows that
    \begin{equation}
        \crn(\gamma_4, \dots, \gamma_{11}) \leq 12. \label{eq:remaining8}
    \end{equation}
    By Lemma~\ref{lem:nopop89}\ref{eq8:lem:nopop89}, we may thus assume that also $\gamma_4, \gamma_5, \gamma_6$ are pairwise disjoint. For the remaining five curves $\gamma_7, \dots, \gamma_{11}$, there are two possibilities. Depending on whether or not they contain yet another triple of pairwise disjoint curves, we get $\crn(\gamma_7, \dots, \gamma_{11}) \geq 4$, with the same reasoning as above or by Lemma~\ref{lem:nopop}. Hence, 
    \begin{align*}
        \crn(\gamma_4, \dots , \gamma_{11}) &\geq 2\cdot 5 + \crn(\gamma_7, \dots, \gamma_{11}) \\
        &\geq 14,
    \end{align*}
    contradicting \eqref{eq:remaining8}. 
\end{proof}

\begin{prop}\label{prop:genus2_1011}
    For $k = 10, 11$, the system of curves $\delta_1, \dots, \delta_k$ given in Figure~\ref{fg:genus2_11curves} is optimal, i.e.\
    \begin{align*}
        \crn(10) &= \ndelta{10} = 21, \\
        \crn(11) &= \ndelta{11} = 28.
    \end{align*}
\end{prop}
\begin{proof}
    For $k=10,11$, let $\gammak{k}$ be a minimal system of $k$ curves.
    Recall that $\ngamma{10}\leq 21$ and $\ngamma{11}\leq 28$. By Lemma~\ref{lem:nonsep}, at least one of the curves must be separating, let us therefore assume that $\gamma_1$ is separating. As in the proof of Proposition~\ref{prop:genus2_89} set $\Gamma = \{\gamma_2, \dots, \gamma_k\}$, ${M=\{\gamma_i \in \Gamma \mid i(\gamma_i, \gamma_1)\neq0\}}$ and $m=\#M$. 
    
    First, consider $k=10$. If $m\geq4$, then by Lemma~\ref{lem:mainestimation}~\ref{eq:mlarge}
    \begin{align*}
        \ngamma{10} &\geq 2m + \crn(9) \\
        &\geq 22,
    \end{align*}
    contrary to $\crn(10) \leq 21$, so $m \leq 3$. By Lemma~\ref{lem:mainestimation}~\ref{eq:k-m-1},
   \allowdisplaybreaks
    \begin{align*} 
        \ngamma{10} &\geq m(9-m)+\crn(m)+ \ntorus{9-m}
        \\
        &= \begin{cases}
            \crn(5;1)+\crn(4;1) &\text{ if $m=0$,}\\
            8 + 2\crn(4;1) &\text{ if $m=1$,}\\
            14 + \crn(4;1)+\crn(3;1) &\text{ if $m=2$,}\\
            18 + 2\crn(3;1) &\text{ if $m=3$,}
        \end{cases} \\
        &= \begin{cases}
            21 &\text{ if $m=0$,}\\
            22 &\text{ if $m=1$,}\\
            24 &\text{ if $m=2$,}\\
            24 &\text{ if $m=3$,}
        \end{cases}        
    \end{align*}
    so $m=0$ and $\crn(10)=21$.

    Let us now consider $k=11$. If $m\geq4$, then
    \begin{align*}
        \ngamma{11} &\geq 2m + \crn(10) \\
        &\geq 29,
    \end{align*}
    by Lemma~\ref{lem:mainestimation}~\ref{eq:mlarge}. This contradicts $\crn(11)\leq 28$, therefore $m \leq 3$. By Lemma~\ref{lem:mainestimation} \ref{eq:k-m-1},
    \begin{align*}
        \ngamma{11} &\geq m(10-m)+\crn(m)+ \ntorus{10-m}
        \\
        &= \begin{cases}
            2\crn(5;1) &\text{ if $m=0$,}\\
            9 + \crn(5;1)+\crn(4;1) &\text{ if $m=1$,}\\
            16 + 2\crn(4;1) &\text{ if $m=2$,}\\
            21 + \crn(4;1)+\crn(3;1) &\text{ if $m=3$,}
        \end{cases} \\
        &= \begin{cases}
            28 &\text{ if $m=0$,}\\
            30 &\text{ if $m=1$,}\\
            30 &\text{ if $m=2$,}\\
            31 &\text{ if $m=3$,}
        \end{cases}        
    \end{align*}
    so $m=0$ and $\crn(11)=28$.
\end{proof}

The proofs of Proposition~\ref{prop:genus2_89} and Proposition~\ref{prop:genus2_1011} give sufficient structural information about minimal systems that uniqueness is an almost immediate consequence. Note that this is not the case for $k \leq 3$ since there are topologically inequivalent triples of pairwise disjoint curves.

\begin{prop} \label{prop:uniqueness11}
    For $4\leq k \leq 11$, minimal systems of $k$ curves in $\Sigma_2$ are unique up to homeomorphisms of the surface and isotopies of curves.
\end{prop}
\begin{proof}
    Let $4 \leq k \leq 11$ and let $\gammak{k}$ be a minimal system of $k$ curves. We prove that $\gammak{k}$ may be mapped to the minimal system $\delta_1, \dots, \delta_k$ (Figure~\ref{fg:genus2_11curves}) by a homeomorphism of $\Sigma_2$ and isotopies of curves.
    
    Recall that, by Lemma~\ref{cor:pop} and Lemma~\ref{lem:nonsep}, one of $\gammak{k}$ is separating. We thus assume that $\gamma_1$ is separating. After a homeomorphism and a homotopy, we may also assume that $\gamma_1 = \delta_1$, where $\delta_1$ is the separating curve in Figure~\ref{fg:genus2_11curves}. Further, for $8 \leq k \leq 11$, it follows from the proof of Proposition~\ref{prop:genus2_89} and Proposition~\ref{prop:genus2_1011} that $m= \#\{\gamma_i \mid 1\leq i \leq k \text{, } i(\gamma_1, \gamma_i) \neq 0\} = 0$, i.e.\ $\gamma_2, \dots, \gamma_k$ do not intersect $\gamma_1$. Moreover, it comes out of the proofs that $\ngamma{k}=\ntorus{k-1}$. A simple computation shows that this is only possible if $\gamma_2, \dots, \gamma_k$ are distributed as evenly as possible between the two connected components of $\Sigma_2 \setminus \gamma_1$. In fact, all of this also holds for $4\leq k \leq 7$, which is straightforward to prove following the same chain of arguments as for $8\leq k \leq 11$. This means that minimal systems consist of a separating curve $\gamma_1$, together with two subsystems of $\left \lceil \frac{k-1}{2} \right \rceil $ and $\left \lfloor \frac{k-1}{2} \right \rfloor $ curves, one in each connected component of $\Sigma_2 \setminus \gamma_1$, which are minimal in terms of the crossing number on the torus. The two connected components of $\Sigma_2 \setminus \gamma_1$ are homeomorphic to punctured tori, to which we would like to apply the uniqueness results from Chapter~\ref{section:genus1}. For this, recall that the mapping class group of the torus is generated by two Dehn twists about simple closed curves intersecting exactly once (Chapter 5 d in \cite{dehn1938gruppe}). Choosing these two curves outside a neighbourhood of the puncture, one sees that for each homeomorphism of the torus, there is a corresponding homeomorphism of the punctured torus. Thus, the uniqueness results from Chapter~\ref{section:genus1} transfer to each connected component of $\Sigma_2 \setminus \gamma_1$, together yielding a homeomorphism of $\Sigma_2$ mapping the curves in each subsystem to the curves corresponding to $w_i$ from \eqref{curves_torus} in Chapter~\ref{section:genus1}. Finally, note that the curves $\delta_i$ in each connected component of $\Sigma_2 \setminus \delta_1$ in Figure~\ref{fg:genus2_11curves} correspond to the curves $w_1, \dots, w_5$ given in \eqref{curves_torus}, yielding the desired result.
\end{proof}

\subsection{Minimal crossing number of 12 curves}

We consider the system of 12 curves in Figure~\ref{fg:genus2_12curves}. In total, there are 36 intersection points. Our goal is to prove that this system is optimal. Notice, how this system of curves looks significantly different than the minimal systems in the previous section: On the one hand, the system no longer contains a separating curve which divides the system into two subsystems. In fact, all curves are non-separating. In particular, this means that this system cannot be obtained from a minimal system of 11 curves by adding a curve. On the other hand, the intersection number of any two curves is at most one, whereas in a minimal system of 11 curves, there are pairs of curves intersecting up to three times. Malestein, Rivin and Theran consider systems of curves with pairwise at most one intersection point, so-called 1-systems, in \cite{Malestein_2013}. They prove that in a genus 2 surface, the maximal number of curves in a 1-system is 12. Thus, the minimal system in Figure~\ref{fg:genus2_12curves} is a maximal 1-system. In fact, we will see that any minimal system of 12 curves is a maximal 1-system.

\begin{figure}[ht] \centering 
        \includegraphics[width=0.8\textwidth]{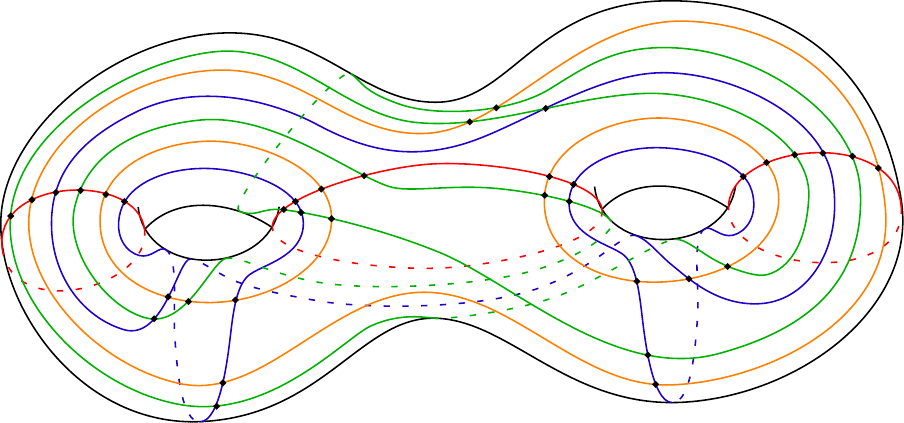}
        \caption{Minimal system of 12 curves}
        \label{fg:genus2_12curves}
\end{figure}

\begin{prop}\label{prop:genus2_12}
    The system of 12 curves in Figure~\ref{fg:genus2_12curves} is optimal, i.e.\ ${\crn(12)=36}$.
\end{prop}    
\begin{proof}
    The system in Figure~\ref{fg:genus2_12curves} gives $\crn(12) \leq 36$. We prove that any minimal system $\gammak{12}$ satisfies $\ngamma{12}\geq36$. For this, we first prove that a minimal system contains no separating curve. Assume towards a contradiction that one of $\gammak{12}$ is separating, w.l.o.g.\ let it be $\gamma_1$. Set $\Gamma = \{\gamma_2, \dots, \gamma_{12}\}$, $M= \{ \gamma_i \in \Gamma \mid i(\gamma_1, \gamma_i) \neq 0 \}$ and $m=\#M$. By Lemma~\ref{lem:mainestimation}~\ref{eq:mlarge}, if $m\geq 5$, then
    \begin{align*}
        \ngamma{12} &\geq 2m + \crn(11) \\
        &\geq 38,
    \end{align*}
    contradicting $\ngamma{12}\leq 36$, so $m\leq 4$. Lemma~\ref{lem:mainestimation}~\ref{eq:k-m-1} yields
    \begin{align*}
        \ngamma{12} &\geq \crn(m)+m(11-m)+\ntorus{11-m} \\
        &= \begin{cases}
            \crn(6;1)+\crn(5;1) &\text{ if $m=0$,} \\
            10 + 2\crn(5;1) &\text{ if $m=1$,} \\
            18 + \crn(5;1)+\crn(4;1) &\text{ if $m=2$,} \\
            24 + 2\crn(4;1) &\text{ if $m=3$,} \\
            1 + 28 + \crn(4;1)+\crn(3;1) &\text{ if $m=4$,} 
        \end{cases} \\
        &= \begin{cases}
            38 &\text{ if $m=0$,} \\
            38 &\text{ if $m=1$,} \\
            39 &\text{ if $m=2$,} \\
            38 &\text{ if $m=3$,} \\
            39 &\text{ if $m=4$,} 
        \end{cases} \\
        &\geq 38,
    \end{align*}
    again contradicting $\ngamma{12} \leq 36$. Therefore, all curves must be non-separating. 
    
     Recall: If $\alpha_1, \dots, \alpha_k$ are non-separating curves and $\alpha_1, \alpha_2, \alpha_3$ are pairwise disjoint, then
    \begin{equation}
        \sum_{i=1}^{3} i(\gamma_i, \gamma_j) \geq 2 \label{eq:alpha}
    \end{equation}
    for $j=4, \dots, k$.
    We will apply this inequality several times.
    
    By Corollary~\ref{cor:nopop_101112}, we may assume that $\gamma_1, \gamma_2, \gamma_3$ are pairwise disjoint since $\ngamma{12} \leq 36$. We consider the crossing number of the remaining nine curves $\gamma_4, \dots, \gamma_{12}$. Inequality~\eqref{eq:alpha} yields
    \begin{align*}
        \crn(\gamma_4, \dots, \gamma_{12}) &= \ngamma{12} - \sum_{j=4}^{12} \sum_{i=1}^3 i(\gamma_i, \gamma_j) \\
        &\leq \ngamma{12} - 2\cdot 9 \\
        & \leq 36-18 \\
        &=18.
    \end{align*}
    Thus, by Lemma~\ref{lem:nopop89}~\ref{eq9:lem:nopop89}, we may also assume that $\gamma_4, \gamma_5, \gamma_6$ are pairwise disjoint. We consider the crossing number of the remaining six curves $\gamma_7, \dots, \gamma_{12}$. If there is yet another triple of pairwise disjoint curves among them, let us assume w.l.o.g.\ that they are $\gamma_7, \gamma_8, \gamma_9$, then
    \begin{align} \begin{split} \label{eq:12curves_6_pop}
    \crn(\gamma_7, \dots, \gamma_{12}) &= \sum_{j=10}^{12} \sum_{i=7}^9 i(\gamma_i, \gamma_j) + \crn(\gamma_{10}, \gamma_{11}, \gamma_{12}) \\
    &\geq \sum_{j=10}^{12} \sum_{i=7}^9 i(\gamma_i, \gamma_j)\\
    &\overset{\mathclap{\eqref{eq:alpha}}}{\geq} 3\cdot2 \\
    & =6.
    \end{split} \end{align}
    If there are no three pairwise disjoint curves among them, then 
    \begin{equation}
    \crn(\gamma_7, \dots, \gamma_{12}) \geq 6 \label{eq:12curves_6_nopop}
    \end{equation}
    by Lemma~\ref{lem:nopop}. Finally, using the above inequalities and applying \eqref{eq:alpha} twice, we get: 
    \begin{align}
        \ngamma{12} &= \sum_{j=4}^{12} \sum_{i=1}^{3} i(\gamma_i, \gamma_j) + \crn(\gamma_4, \dots, \gamma_{12}) \nonumber\\      
        &\geq  9\cdot 2 + \crn(\gamma_4, \dots, \gamma_{12}) \label{eq:12curves_1st_pop} \\
        &= 18 + \sum_{j=7}^{12} \sum_{i=4}^{6} i(\gamma_i, \gamma_j) + \crn(\gamma_7, \dots, \gamma_{12}) \nonumber\\
        &\geq 18 + 6 \cdot 2 + \crn(\gamma_7, \dots, \gamma_{12}) \label{eq:12curves_2nd_pop} \\
        &\geq 18+12+6 \label{eq:12curves_3rd_pop}\\
        &= 36. \nonumber \qedhere
    \end{align}
\end{proof}

The above proof gives a lot of information about what conditions minimal systems of 12 curves necessarily satisfy. We will make use of that to deduce that minimal systems of 12 curves are unique. For this, we study some of these properties more closely.

\begin{lem}\label{lem:inr1} \, 
     \begin{enumerate}[label=(\roman*)]
       \item \label{lem:inr1_item1} Let $\gamma_1, \gamma_2, \gamma_3, \delta_1, \delta_2 \subset \Sigma_2$ be a system of non-separating curves such that $\delta_1, \delta_2$ are disjoint, $\gamma_1, \gamma_2, \gamma_3$ are pairwise disjoint, and
         \begin{align*}
            \sum_{i=1}^3 i(\gamma_i, \delta_1)=\sum_{i=1}^3 i(\gamma_i, \delta_2)=2.
        \end{align*}
         Then all pairwise intersection numbers are at most~1.
        \item \label{lem:inr1_item2} Let $\gamma_1, \gamma_2, \gamma_3, \delta_1, \delta_2, \delta_3 \subset \Sigma_2$ be a systems of non-separating curves such that $\gamma_1, \gamma_2, \gamma_3$ and $\delta_1, \delta_2, \delta_3$ are pairwise disjoint and satisfy:
        \begin{align*}
            \sum_{i=1}^3 i(\gamma_i, \delta_j) = 2 
        \end{align*}
        for $j=1,2,3$. Then all pairwise intersection numbers are at most~1.
     \end{enumerate}
\end{lem}
\begin{proof}
    \ref{lem:inr1_item1}: It suffices to prove that $i(\gamma_i, \delta_j) \leq 1$ for $i=1,2,3$ and $j=1,2$ since all other intersection numbers are 0. By assumption, $i(\gamma_i, \delta_1)\leq2$ and $i(\gamma_i, \delta_2) \leq 2$ for $i=1,2,3$, so assume towards a contradiction that one of the intersection numbers is 2, w.l.o.g.\ $i(\gamma_1, \delta_1)=2$. By the conditions on the sum of the intersection numbers, it follows that $i(\gamma_2, \delta_1)=i(\gamma_3, \delta_1)=0$. Since $\gamma_1$ and $\delta_1$ are non-separating, intersect exactly twice and are disjoint from $\gamma_2$ and $\gamma_3$, $\Sigma_2\setminus (\gamma_1 \cup \delta_1)$ is the disjoint union of two annuli containing the curves $\gamma_2$ and $\gamma_3$. Therefore, $\delta_2$ cannot lie in $\Sigma_2\setminus (\gamma_1 \cup \delta_1)$ and hence must intersect $\gamma_1$ since $\delta_2 \cap \delta_1 = \emptyset$. By the conditions on the sum of intersection numbers, it follows that either $i(\gamma_2, \delta_2)=0$ or $i(\gamma_3, \delta_2)=0$. Assume w.l.o.g\ that $i(\gamma_2, \delta_2)=0$. Then $\gamma_3, \delta_2 \subset \Sigma_2 \setminus (\gamma_2 \cup \delta_1)$, which is homeomorphic to a sphere with four removed discs since $\gamma_2$ and $\delta_1$ are non-separating and $\gamma_2 \cap \delta_1 = \emptyset$. It is straightforward to see that in a sphere with four boundary components any two distinct non-trivial simple closed curves (not homotopic to the boundary) intersect and are separating. Thus, $i(\gamma_3, \delta_2)\geq2$ and therefore $\sum_{i=1}^3 i(\gamma_i, \delta_2) \geq 3$, contradicting the conditions given in the lemma.\\
    \ref{lem:inr1_item2}: Apply \ref{lem:inr1_item1} to $\gamma_1, \gamma_2, \gamma_3$ and $\delta_1, \delta_2$ and to $\gamma_1, \gamma_2, \gamma_3$ and $\delta_2, \delta_3$.
\end{proof}

\begin{cor}
    Any system of 12 curves with minimal crossing number is a maximal 1-system, i.e.\ all pairwise intersection numbers are at most~1.
\end{cor}
\begin{proof}
    Let $\gammak{12}$ be a minimal system, i.e.\ $\ngamma{12}=36$, and set $A=\{\gamma_1, \gamma_2, \gamma_3\}$, $B=\{\gamma_4, \gamma_5, \gamma_6\}$, $C=\{\gamma_7, \gamma_8, \gamma_9\}$, $D=\{\gamma_{10}, \gamma_{11}, \gamma_{12}\}$.    
    It follows from the proof of Proposition~\ref{prop:genus2_12} that all curves must be non-separating. Further, we may assume that $A$ and $B$ are two triples of pairwise disjoint curves. Since $\ngamma{12}=36$, we get equality in all the inequalities in the proof of Proposition~\ref{prop:genus2_12}. Equality in \eqref{eq:12curves_1st_pop} yields
    \begin{align*}
        i(\gamma_1, \gamma_i)+i(\gamma_2, \gamma_i)+i(\gamma_3, \gamma_i) = 2
    \end{align*}
    for $i=4, \dots, 12$, and equality in \eqref{eq:12curves_2nd_pop} gives
    \begin{align*}
        i(\gamma_4, \gamma_i)+i(\gamma_5, \gamma_i)+i(\gamma_6, \gamma_i) = 2
    \end{align*}
    for $i=7, \dots, 12$. Equality in \eqref{eq:12curves_3rd_pop} means
    \[\crn(\gamma_7, \dots, \gamma_{12}) = 6,\]
    which is possible in one of two cases; we consider them separately. 
    
    In the first case, the above equality corresponds to equalities in \eqref{eq:12curves_6_pop}, where we may assume that $C$ is a third triple of pairwise disjoint curves. These equalities imply that 
    \begin{align*}
        i(\gamma_7, \gamma_i)+i(\gamma_8, \gamma_i)+i(\gamma_9, \gamma_i) = 2
    \end{align*}
    for $i=10,11,12$, and that $\crn(\gamma_{10}, \gamma_{11}, \gamma_{12})=0$, i.e.\ also the curves in $D$ are pairwise disjoint. This means that any two of $A, B, C, D$ satisfy the conditions of Lemma~\ref{lem:inr1} \ref{lem:inr1_item2}, it therefore follows that all intersection numbers are at most~1.

    In the second case, the above equality corresponds to equality in \eqref{eq:12curves_6_nopop}, in which case there are no three pairwise disjoint curves among $\gamma_7, \dots, \gamma_{12}$. Now observe that we applied Lemma~\ref{lem:nopop} to obtain this inequality, which we proved by considering intersection graphs and applying Turán's theorem. Returning to the proof of that lemma, one sees that we get equality if and only if we have equality in \eqref{eq:nrofedges} and \eqref{eq:turanstheorem}. The first one implies $i(\gamma_i, \gamma_j)\leq1$ for $7\leq i,j \leq 12$, i.e.\ all pairwise intersection numbers of curves in $C\cup D$ are at most~1. The second one yields, by Turán's theorem (Theorem 12.3 in \cite{bondy2008graph}), that the intersection graph is isomorphic to the following graph:
    \begin{center}
    \begin{tikzpicture}[main/.style = {draw, circle}, state/.style={circle, draw, minimum size=0.75cm}] 
    \node[state] (1) {$8$}; 
    \node[state] (2) [above right of=1] {$7$};
    \node[state] (3) [below right of=1] {$9$}; 
    \node[state] (4) [right of=2] {$12$};
    \node[state] (5) [below right of=4] {$11$};
    \node[state] (6) [below left of=5] {$10$};
    \draw (1) -- (2);
    \draw (2) -- (3);
    \draw (1) -- (3);
    \draw (4) -- (5);
    \draw (5) -- (6);
    \draw (4) -- (6);
    \end{tikzpicture}
    \end{center}
    Thus, we may assume that $C$ and $D$ each consist of curves intersecting pairwise exactly once, and all other intersection numbers are zero. Applying Lemma~\ref{lem:inr1} \ref{lem:inr1_item2} to $A$ and $B$, yields that all pairwise intersection numbers of curves in $A\cup B$ are at most~1. Finally, applying Lemma~\ref{lem:inr1} \ref{lem:inr1_item1} to $A$ and any choice of $\gamma_i \in C$ and $\gamma_j \in D$ gives that all pairwise intersection numbers of curves in $A \cup C \cup D$ are at most~1. Analogously we get the same result for $B \cup C \cup D$.
\end{proof}

\begin{cor} \label{cor:uniqueness12}
    In $\Sigma_2$, minimal systems of 12 curves are unique up to homeomorphisms of the surface and isotopies of curves. 
\end{cor}
\begin{proof}
    Theorem 1.3 in \cite{Malestein_2013} states that for a genus 2 surface, there are two mapping class orbits of maximal 1-systems. In the proof, the systems are described as isomorphism classes of planar graphs on 6 vertices. It arises from this description that the two maximal 1-systems have crossing numbers 36 and 38, respectively. Since minimal systems of 12 curves are 1-systems with crossing number 36, it thus follows that they are unique up to homeomorphisms of $\Sigma_2$ and isotopies of curves.
\end{proof}

\begin{rem}
With its 12 systoles, the Bolza surface realises the maximal number of systoles among hyperbolic surfaces of genus 2.
The description of the Bolza surface in Chapter 5 of \cite{schmutz1993riemann} implies that the system of its systoles has crossing number 36. Thus, the system of systoles is a minimal system of 12 curves and is therefore equivalent to the one given in Figure~\ref{fg:genus2_12curves}.
 \end{rem}


\section{Higher genus} \label{section:higher genus}

In the last two chapters, we have seen that in surfaces of genus 1 and 2, small minimal systems are unique up to homeomorphisms of the surface and isotopies of curves. The only exceptions are systems of at most three curves in a genus 2 surface, where uniqueness is lost due to inequivalent pair of pants decompositions of the surface. We now briefly look at minimal systems on surfaces of genus $g\geq3$. We will see that uniqueness is lost in its strongest sense, preserved in a weakened form in some cases and completely lost in other cases.

It is clear that realisations of $\crn(k;g)$ for $k\leq 3g-3$ are not unique, due to the existence of inequivalent decompositions of $\Sigma_g$ into pairs of pants. We therefore consider $k\geq 3g-2$.

Recall from Proposition~\ref{prop:smallkgeneralg} that $\crn(3g-3+k;g)=k$ for $1 \leq k \leq g$. In the proof of that proposition, we have explored what this means for the intersection pattern of optimal systems: Any curve $\gamma_i$ in a minimal system $\gammak{3g-3+k}$ satisfies
\[\sum_{j=1}^{3g-3+k} i(\gamma_i, \gamma_j) \leq \crn(3g-3+k;g) - \crn(3g-3+k-1; g) \leq 1,\]
which means that $\gammak{3g-3+k}$ contains $k$ disjoint pairs of curves intersecting exactly once, while the remaining $3g-3-k$ curves are disjoint from all other curves. These $3g-3-k$ remaining curves are thus contained in the complement of the $k$ pairs of curves in $\Sigma_g$, which is homeomorphic to a surface of genus $g-k$ with $k$ removed discs. It is clear that $k$ of those curves must be homotopic to the boundaries of the removed discs, else the system would not be minimal. The other $3g-3-2k$ curves define a pair of pants decomposition of the complement of the $k$ pairs of curves in $\Sigma_g$. With the exception of $g=k=3,4,5$, there are inequivalent choices of these curves. Thus, besides these exceptions,  realisations of $\crn(3g-3+k;g)$ are not unique up to homeomorphisms of $\Sigma_g$ and isotopies of curves. However, they are unique up to the choice of these remaining $3g-3-k$ curves. 

Finally, recall from Proposition~\ref{prop:smallkgeneralg} that $\crn(4g-3+k;g)=g+2k$ for $1\leq k \leq g$. Looking at the minimal systems given in the proof of that proposition (see Figure~\ref{fg:3g-3+k}), again it is clear that for $g \geq 6$, there are topologically inequivalent choices of curves $\alpha_1, \dots, \alpha_{2g-3}$. One may think that those realisations of ${\crn(4g-3+k;g)}$ are unique up to the choice of the curves $\alpha_1, \dots, \alpha_{2g-3}$. We will illustrate on the example of a genus 4 surface that this is not the case. In Figure~\ref{fg:lack_of_uniqueness}, there are two systems of 14 curves with 6 intersection points. By Proposition~\ref{prop:smallkgeneralg}, $\crn(14;4)=6$, those systems are therefore minimal. The two systems coincide with the exception of one curve, they are only distinguished by the green curve. In the system on the left-hand side, the green curve is non-separating, while it is separating in the system on the right-hand side. These systems are thus different in any regard. Adding curves to those systems gives distinct minimal systems of $k\leq 17$ curves.
The same construction yields inequivalent minimal systems of $k$ curves for $g \geq 4$ and $4g-2 \leq k \leq 5g-3$.

\begin{figure}[ht] \centering 
        \includegraphics[width=0.8\textwidth]{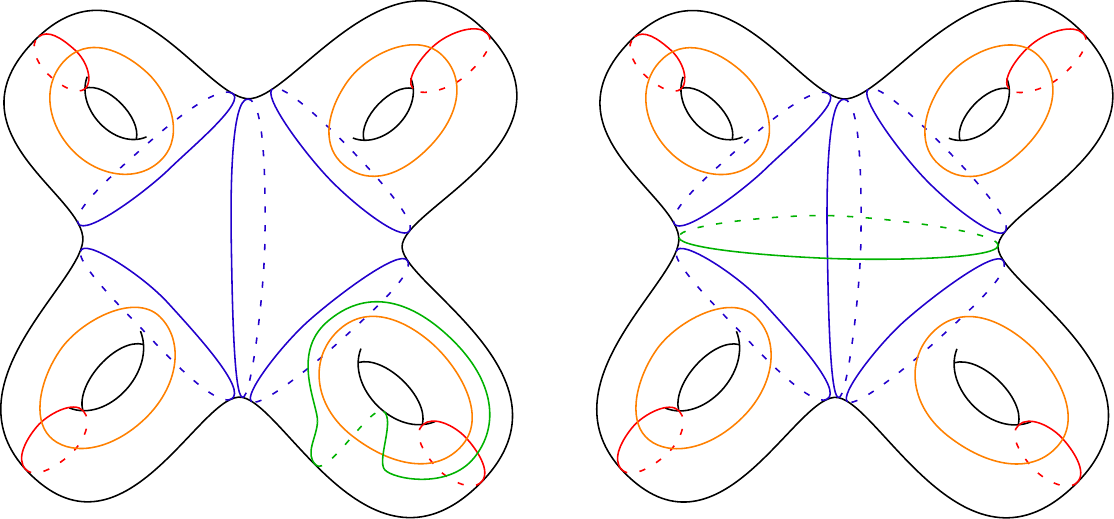}
        \caption{Two different minimal systems of 14 curves in a surface of genus 4}
        \label{fg:lack_of_uniqueness}
\end{figure}

\section{Order of growth} \label{section:order of growth}

We conclude by considering the behaviour of $\crn(k;g)$ for large $k$, giving a lower bound of order $k^2$ and an upper bound of order $k^{2+\frac{1}{3g-3}}$. The lower bound is a straightforward consequence of the fact that any subsystem of $3g-2$ curves has crossing number at least 1. The upper bound is obtained by choosing a hyperbolic metric on the surface and considering the crossing number of all simple closed geodesics of length not exceeding $L$.

\begin{prop}
    For $g \geq 2$, there exists a constant $c_g>0$ such that for any $k \geq 3g-2$:
    \[\crn(k; g) \geq c_g k^2.\]
\end{prop}
\begin{proof}
    Use Corollary~\ref{cor:count} with $m=3g-2$ and the fact that, by Proposition~\ref{prop:smallkgeneralg}, $\crn(3g-2;g)=1$: 
     \begin{align*}
        \crn(k; g) &\geq \frac{k(k-1)}{(3g-2)(3g-3)} \crn(3g-2;g)\\
        &= \frac{k(k-1)}{(3g-2)(3g-3)} \\
        &\geq c_g k^2,
    \end{align*}
    where $c_g = \frac{1}{2(3g-2)(3g-3)}$.
\end{proof}

\begin{rem}
    A lower bound for $\crn(k;g)$ of order $k^{2+\frac{1}{6g-6}}$ is obtained by Hubard and Parlier. Their article containing this result is currently in preparation \cite{parlier2024}.
\end{rem}

\begin{prop}
    For $g \geq 2$, there exist constants $k_g, c_g > 0$ such that for any $k\geq k_g$:
    \[\crn(k;g) \leq c_g k^{2+\frac{1}{3g-3}}.\]  
\end{prop}
\begin{proof}
    Choose a hyperbolic metric on $\Sigma_g$ with a simple simple length spectrum, that is, a metric such that all lengths of simple closed geodesics occur with multiplicity 1, see \cite{McShane_2008} for the existence of such a metric. For $L>0$, let $\mathcal{S}(L)$ denote the set of simple closed geodesics of length at most $L$. The following limit is known to exist due to Mirzakhani \cite{mirzakhani2008growth}:
    \[\lim_{L \to \infty} \frac{\#S(L)}{L^{6g-6}} =c > 0.\]
    In particular, there exist positive constants $a_g$ and $L_g$ such that for $L \geq L_g$,
    \[\frac{1}{a_g} L^{6g-6} \leq \# \mathcal{S}(L) \leq a_g L^{6g-6},\]
    implying that
    \begin{equation}
        L^2 \leq (a_g \#\mathcal{S}(L))^{\frac{1}{3g-3}}. \label{eq:L2}
    \end{equation}
    Denoting by $\ell(\gamma)$ the length of a curve $\gamma$, there is a constant $b_g>0$ such that for any two simple closed curves $\gamma$ and $\delta$,
    \begin{equation*}
        i(\gamma, \delta) \leq b_g \ell(\gamma)\ell(\delta)
    \end{equation*}
    (Lemma 4.2 in \cite{fathi2012thurston}). For $\gamma, \delta \in \mathcal{S}(L)$, this implies
    \begin{equation}
        i(\gamma, \delta) \leq b_g L^2. \label{eq:torkaman}
    \end{equation}
    Define $k_g=\# \mathcal{S}(L_g)$ and let $k\geq k_g$. Since the simple length spectrum is simple, there exists $L \geq L_g$ such that $k = \#\mathcal{S}(L)$. Thus, we get:
    \begin{align*}
        \crn(k;g) &\leq  \crn(\mathcal{S}(L)) \\
        &\stackrel{\mathclap{\eqref{eq:torkaman}}}{\leq}  \binom{k}{2} b_g L^2 \\
        &\leq  \frac{1}{2} b_g k^2 L^2 \\
        &\stackrel{\mathclap{\eqref{eq:L2}}}{\leq}  \frac{1}{2} b_g a_g^{\frac{1}{3g-3}} k^{2+\frac{1}{3g-3}}\\
        &=  c_g k^{2+\frac{1}{3g-3}},
    \end{align*}
    where $c_g = \frac{1}{2} b_g a_g^{\frac{1}{3g-3}}$.
\end{proof}

\section{A brief outlook} \label{section:outlook}

The system of 12 curves considered in Section~\ref{section:genus2} is optimal in three ways: It minimises the crossing number of 12 curves, it is a maximal 1-system, and it is equivalent to the maximal system of systoles in the Bolza surface. It would be worth further investigation, whether these optimal systems are related in general.

\begin{ques}
    In a hyperbolic surface with the maximal number of systoles, is the system of systoles minimal in terms of the crossing number?
\end{ques}

\begin{ques}
    In a surface of genus $g$, does there exist a maximal 1-system that minimises the crossing number?
\end{ques}

More generally, we have considered the set $\mathcal{S}(L)$ of all geodesics of length at most $L$ in Section~\ref{section:order of growth}, which provided an upper bound for $\crn(k;g)$ of order $k^{2+\frac{1}{3g-3}}$. This leads to the following question.

\begin{ques}
    In a hyperbolic surface of genus $g$, is the crossing number of the system $\mathcal{S}(L)$ of all geodesics of length at most $L$ asymptotically proportional to $\# \mathcal{S}(L)^{2+\frac{1}{3g-3}}$? 
    Explicitly, does the following statement hold:
    \[\lim_{L \to \infty} \frac{\log( \crn(\mathcal{S}(L)))}{\log(\# \mathcal{S}(L))} = 2+\frac{1}{3g-3} \ ?\]
\end{ques}

During the process of writing this article, Hubard and Parlier obtained a lower bound for $\crn(k;g)$ of order $k^{2+\frac{1}{6g-6}}$; their article containing this result is in preparation \cite{parlier2024}. The precise order of growth remains open. 

\newpage
 \bibliographystyle{alpha}
 \bibliography{bibliography}
 
\end{document}